\documentclass[12pt]{article}
\usepackage{graphicx,amsmath,amsthm,amssymb,amscd,color,epsfig,geometry,graphicx,setspace}
\usepackage{fancyhdr}
\usepackage{diagrams}
\newtheorem{theorem}{Theorem}
\newtheorem{lemma}[theorem]{Lemma}
\newtheorem{proposition}[theorem]{Proposition}
\newtheorem{corollary}[theorem]{Corollary}

\begin{document}

\pagenumbering{roman} \thispagestyle{empty}

\pagebreak

\begin{center}
\large{HYPERBOLIC GRAPHS OF SURFACE GROUPS}
\end{center}
\begin{abstract}

We give a sufficient condition under which the fundamental group
of  a reglued graph of surfaces is hyperbolic. A reglued graph of surfaces is constructed by cutting a fixed graph of surfaces along the edge surfaces, then regluing by pseudo-Anosov homeomorphisms of the edge surfaces. By carefully choosing the regluing homeomorphism, we
construct an example of such a reglued graph of surfaces, whose
fundamental group is not abstractly commensurate to any
surface-by-free group, i.e., which is different from all the
examples given in the paper \cite {Mosher:hypbyhyp}.

\end{abstract}

\pagebreak

\pagebreak

\tableofcontents


\pagebreak

\pagenumbering{arabic} \pagestyle{myheadings} \markboth{}{}
\section{Introduction} The fundamental group of the mapping torus of a
pseudo-Anosov homeomorphism of an oriented closed hyperbolic
surface is hyperbolic. This was first proved by Thurston. A direct
proof was given by Bestvina and Feighn \cite{BestvinaFeighn:combination}. Using their idea, Mosher   \cite{Mosher:hypbyhyp} proved the following theorem.

 Consider an oriented closed hyperbolic surface $S$. Let $\Phi_{1},\cdots,\Phi_{ m}
\in MCG(S)$ be an independent set of pseudo-Anosov mapping classes
of $S$, and let $\phi_{1},\cdots,\phi_{ m} \in Homeo(S)$ be
pseudo-Anosov representatives of $\Phi_{1},\cdots,\Phi_{m}$
respectively. If $i_1$, $\cdots$, $i_m$ are large enough positive
integers, then the fundamental group of the graph of spaces
$\mathcal G$, as shown in Figure 1,
\begin{figure}[h]
\begin{center}
\includegraphics{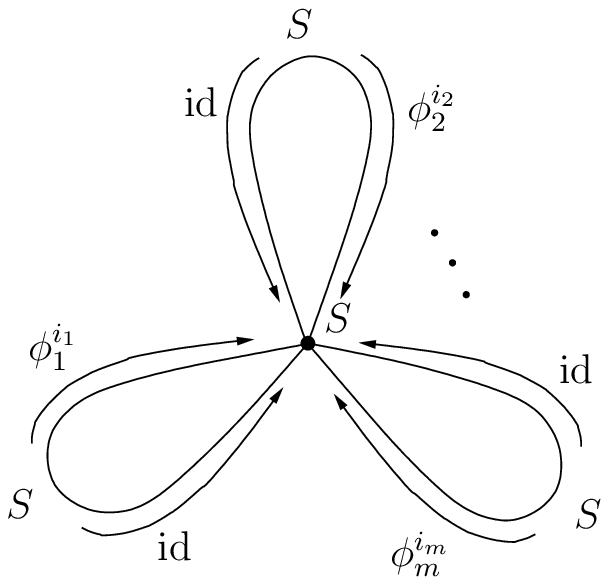}
\end{center}
\caption{}
\end{figure}
is a hyperbolic group. In the statement of this theorem, by saying
a set $B$ of pseudo-Anosov mapping classes is \emph{independent},
we mean the sets $Fix(\Phi)$ are pairwise disjoint for $\Phi\in
B$, where $Fix(\Phi)$ consists of the attractor and the repeller
of $\Phi$ on the space of projective measured foliations
$\mathcal{PMF}(S)$.



 A \emph{graph of surfaces} $S\Gamma$ consists of an  oriented
connected finite underlying graph $\Gamma$, a function which
assigns to each vertex a closed hyperbolic surface or orbifold, to
each edge a closed hyperbolic surface, and another function which
assigns to each oriented edge a covering map from the edge surface
to the vertex surface of the  origin of the edge. In the cases
studied in this paper, we change the canonical graph of surfaces
by cutting along the edge surfaces, then choosing pseudo-Anosov
homeomorphisms of the edge surfaces, then regluing. We call it a
\emph{graph of surfaces with pseudo-Anosov regluing}.
 Thus the mapping torus of a
pseudo-Anosov homeomorphism can be considered as this type of
space whose underlying graph consists of only one vertex and one
edge, and the vertex and edge spaces are the same hyperbolic
surface. The case studied by Mosher is another reglued graph of
surfaces with the underlying graph consists of only one vertex, in
addition the vertex and edge spaces
 are the same hyperbolic surface $S$.

We shall extend Mosher's theorem to the general graphs of surfaces
with pseudo-Anosov regluing. Theorem 1 says that if the
pseudo-Anosov homeomorphisms are chosen to satisfy an appropriate
independence condition, then the fundamental group of the reglued
graph of surfaces is word hyperbolic, when these homeomorphisms
are replaced with sufficiently high powers of themselves.

We shall describe this cutting and regluing process with more
details. Let $S\Gamma$ be a graph of surfaces with the underlying
graph $\Gamma$, let $E$ be the set of oriented edges of $\Gamma$,
and let $V$ be the set of vertices of $\Gamma$. For each $e\in E$,
let $S_e$ be the corresponding edge surface. For each oriented
edge $e$, there is a finite covering map $p_e: S_e\rightarrow
F_{o(e)}$, where $F_{o(e)}$ is the vertex surface of the origin
$o(e)$ of the
 edge $e$.  For each inverse pair
of oriented edges $e,\ \overline{e}$, there is an inverse pair
of homeomorphisms $g_e: S_e\rightarrow S_{\overline{e}},\
g_{e}^{-1}: S_{\overline{e}}\rightarrow S_{e}$. Let
$\boldsymbol{\varphi}=\{\phi_e|\ e\in E\}$, where $\phi_e:
S_e\rightarrow S_e$ is a pseudo-Anosov homeomorphism of $S_e$. Let $S\Gamma_{\boldsymbol{\varphi}}$
be the graph of surfaces with pseudo-Anosov regluing obtained from
$S\Gamma$ by cutting along each $S_e$ and regluing using $\phi_e$,
i.e., in the reglued graph of surfaces, the effect is to replace
the map $g_e: S_e\rightarrow S_{\overline e}$ by the map
$g_e\circ\phi_e$, for $e\in E$. Moreover, let
$\boldsymbol{m}=\{m_e|\ e\in E\}$, where $m_e$ are positive
integers, and let $S\Gamma_{\boldsymbol{\varphi^m}}$ be the graph
of surfaces obtained from $S\Gamma$ by regluing using
$\phi_e^{m_e}$ for each $e\in E$.

Given a vertex $v$ of the underlying graph $\Gamma$, let $F_v$ be
the corresponding vertex surface (or orbifold). For each $v\in V$,
denote $I_v=\{i|\ e_i \text{ is an oriented edge such that the}$  
$\text{origin}$ $ \text{of } e_i \text{ is } v\}$.  For each $v\in V$ and
each $i\in I_v$, there is a finite index covering map
$p_i:S_{i}\rightarrow F_v$, where $S_i$ is a shorthand notation of
$S_{e_i}$. For an oriented edge $e_i$ has the vertex $v$ as both
of its origin and terminal, the covering maps
$S_i\xrightarrow{p_i} F_v$ and $S_i\xrightarrow{g_i}S_{\overline
i}\xrightarrow{p_{\overline{i}}} F_v$ might be different, where $g_i$ is a shorthand notation of
$g_{e_i}$.
The portion of $S\Gamma_{\boldsymbol{\varphi^m}}$ around a vertex
space $F_v$ could look like in Figure 2
\begin{figure}[h]
\begin{center}
\input{general.pstex_t}
\end{center}
\caption{}
\end{figure}
\bigskip

For the purpose of Theorem 1, fix a hyperbolic structure on each
vertex surface $F_v$. For each $v\in V$ and each $i\in I_v$,
suppose $S_i$ equipped with the pullback metric by the covering
map $p_i:S_i\rightarrow F_v$. Hence for each covering map $p_i$,
there is the derivative map $Dp_i:PS_i\rightarrow PF_v$, where
$PS_i$,  $PF_v$ are the projective tangent bundles of $S_i$ and
$F_v$ respectively. For an oriented edge $e_j$, let
$\phi_{j}^{m_j}:S_{j}\rightarrow S_{j}$ be the pseudo-Anosov
homeomorphism for the edge $e_j$, with the stable geodesic
lamination $\Lambda^s_j\subset S_{j}$; the stable geodesic
lamination $\Lambda^s_{\overline j}\subset S_{\overline j}$ of
$(\phi^{m_j}_{\overline j})=g_j\phi_j^{-m_j}g_j^{-1}$ is homotopic
to the image under $g_j$ of the unstable geodesic lamination of
$\phi^{m_j}_j$. The geodesic laminations $\Lambda^s_j$ and
$\Lambda^s_{\overline j}$ are independent of the choice of the
exponent $m_j$. In the following, let $T\Lambda_i^s$ denote the
unit tangent vector space of $\Lambda^s_i$.

The main theorem of this paper is
\begin{theorem}
 Let $S\Gamma_{\boldsymbol{\varphi^m}}$ be a graph of surfaces with pseudo-Anosov regluing. Let $\Gamma$ be  its underlying
  graph . If for each vertex $v\in \Gamma$, and for each $i\in
  I_v$,
 the derivative maps
  $Dp_i|T\Lambda^s_i$ are injections, and their images are disjoint compact subsets of
 $PF_v$, then the fundamental group of $S\Gamma_{\boldsymbol{\varphi^m}}$ is
hyperbolic, when $m_i\in \boldsymbol{m}$ are sufficiently large.
\label{main}
\end{theorem}

The proof of the hyperbolicity of
$S\Gamma_{\boldsymbol{\varphi^m}}$ depends ultimately on the
Combination Theorem of \cite{BestvinaFeighn:combination}.  The
Combination Theorem says that if the quasi-isometrically embedded
condition (which is automatically satisfied in the cases studied
in this paper) and the hallways flare condition (which is much
more difficult to check) both hold, then
$S\Gamma_{\boldsymbol{\varphi^m}}$ is a hyperbolic space.  In
order to check the satisfaction of the hallways flare condition,
we need to extend the parallel corresponds lemma
\cite{Mosher:hypbyhyp}, the key in that paper, to a new version of
the parallel corresponds lemma.

The idea of the proof of Theorem 1 is: by applying the new version
of parallel corresponds lemma, if the hypothesis of Theorem 1 is
satisfied, then the hallways flare condition is satisfied.
Therefore the fundamental group of
$S\Gamma_{\boldsymbol{\varphi^m}}$ is hyperbolic.

Here are some applications of this theorem.

First: let $S$ be a closed hyperbolic surface, let $G$, $H$ be
finite subgroups of the mapping class group $MCG(S)$, and let
$\Phi\in MCG(S)$ be a pseudo-Anosov mapping class. Suppose $G$,
$H$ each have trivial intersection with the virtual centralizer of
$\langle\Phi\rangle$ in $MCG(S)$, then for sufficiently large $n$,
the subgroup $A$ of $MCG(S)$ generated by $G,\Phi^nH\Phi^{-n}$ is
isomorphic to the free product of these subgroups. Even more,
 $A$ is a virtual Schottky subgroup of ${MCG(S)}$, in the sense of \cite{FarbMosher:quasiconvex}.

Second: let $\mathcal G_{\phi^m}$ be a graph of surfaces with regluing as in
Figure 3,
\begin{figure}[h]
\begin{center}
\input{spacegraph.pstex_t}
\end{center}
\caption{}
\end{figure}
where $S$, $F$ are genus 3 and 2 tori, $\phi: S\rightarrow S$ is a
pseudo-Anosov homeomorphism. Suppose there exist simple closed
curves $a\subset F$ and $c\subset S$, as shown in Figure 4, such
that $p^{-1}(a)=c,\ c\subset q^{-1}(a)$, and $q^{-1}(a)$ is
disconnected. In addition, suppose that in the group $MCG(S)$, the
virtual centralizer of $\langle \Phi \rangle$ has trivial
intersection with the deck transformation groups of $p$ and $q$,
where $\Phi$ is the mapping class of $\phi$. Then $\pi_1(\mathcal
G_{\phi^m})$ is hyperbolic when m is sufficiently large.
\begin{figure}[h]
\begin{center}
\includegraphics{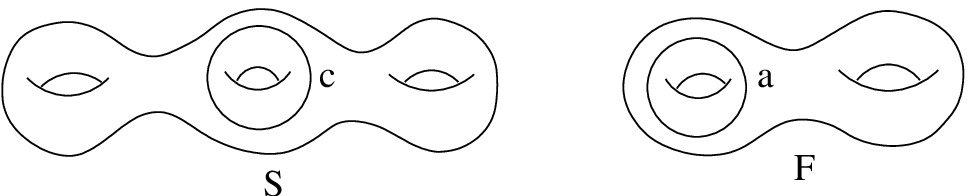}
\end{center}
\caption{}
\end{figure}

More interesting, we will see that there exists a pseudo-Anosov
homeomorphism $\phi$ of $S$, such that $\pi_1(\mathcal
G_{\phi^m})$ is not commensurate to $\pi_1(S^\prime)\rtimes K$,
for any oriented, closed hyperbolic surface $S^\prime$, and for
any free group $K$, where $\mathcal G_{\phi^m}$ as the above.
More than that, $\pi_1(\mathcal G_{\phi^m})$ is not even
quasi-isometric to any surface-by-free group. Therefore
$\pi_1(\mathcal G_{\phi^m})$ is different from all the hyperbolic
groups constructed in \cite{Mosher:hypbyhyp}.

\emph{Problems.} Do there exist some reducible homeomorphisms of
the edge surfaces, such that the graph of surfaces with reducible
homeomorphism regluing are hyperbolic?

Is Theorem 1 still true when the vertex and edge groups are free
groups?
\bigskip


\section{Preliminaries}

In this section, we recall some preliminaries about combinatorial
and geometric group theory, and some facts of hyperbolic geometry
which will be used later.
\bigskip

\noindent\textbf{Graphs of surfaces}  (The
material in this subsection can be found in \cite{ScottWall} and
\cite{Serre:trees})

Let $\Gamma$ be a connected finite graph, let $e$ be an oriented
edge of $\Gamma$, and let $\overline e$ be the inverse edge of
$e$. The vertex $o(e)$ is called the origin of $e$ and the vertex
$t(e)$ is called the terminal of $e$, obviously $o(e)=t(\overline
e)$.

A \emph{graph of surfaces} $S\Gamma$ consists of a connected
finite graph $\Gamma$ and a function which assigns to each vertex
$v\in \Gamma$ a closed hyperbolic surface or orbifold $F_v$, to
each pair of oriented edges $e$, $\overline e$ closed hyperbolic
surfaces $S_e$, $S_{\overline e}$ and an inverse pair of
homeomorphisms $S_e\rightarrow S_{\overline e}$, $S_{\overline
e}\rightarrow S_{e}$, and to each edge $e$ a continuous map
$p_{e}:S_e\rightarrow F_{o(e)}$, such that $p_e$ induces an
injection on the fundamental groups. In most of our cases $p_e$
are covering maps for every edge $e$ of $\Gamma$.

Given a graph of surfaces $S \Gamma$, we can define the
\emph{total space} $S_\Gamma$ as the quotient of the disjoint
union $(\cup\{F_v| v\in V(\Gamma)\})\bigcup(\cup\{S_e\times I|
e\in E(\Gamma)\})$ by identifying the equivalent classes:
$(s,0)\sim p_e(s)$ for $(s,0)\in S_e\times 0$, $p_e(s)\in F_{o(e)}
$; $(s,1)\sim p_{\overline e}(s)$ for $(s,1)\in S_e\times 1$,
$p_{\overline e}(s)\in F_{t(e)}$. The \emph{fundamental group of
the graph of surfaces} $\pi_1(S\Gamma)$ is defined to be the
fundamental group of the total space $S_\Gamma$. There is a
projection map $\pi:S_{\Gamma}\rightarrow \Gamma$ such that each
vertex surface $F_v$ maps to the vertex $v$ and each $S_e\times I$
maps to the edge $e$, $\pi$ is an onto map.

The universal cover $\widetilde {S\Gamma}$ of $S\Gamma$ is a union
of copies of the universal covers $\widetilde {S_e}\times I$ and
$\widetilde {F_v} $. In $\widetilde {S\Gamma}$, if we identify
each copy of $\widetilde {F_v}$ to a point and each copy of
$\widetilde {S_e}\times I$ to a copy of $I$, then we obtain a
graph $t$ and there is a canonical projection map $\widetilde \pi:
\widetilde{S\Gamma}\rightarrow t$. It is not hard to see that $t$
is a tree, called the \emph{Bass-Serre tree}. The action of
$\pi_1(S\Gamma)$ on $\widetilde{S\Gamma}$ descends to an action of
$\pi_1(S\Gamma)$ on $t$, where the quotient graph coincides with
the original graph $\Gamma$, and the stabilizers of each vertex
and each edge of $t$ are conjugates of corresponding fundamental
groups of $F_v$ and $S_e$.

\bigskip
\noindent\textbf{The Bestvina -- Feighn Combination Theorem}  (The material in this subsection can be found in \cite{BestvinaFeighn:combination})

For the purpose of this paper, instead of the original combination
theorem, we shall state the tailored Bestvina-Feighn Combination Theorem
in the context of the graphs of surfaces only.

 Let $S\Gamma$ be a graph of surfaces with the underlying graph
 $\Gamma$, and let $\pi: S\Gamma \rightarrow \Gamma$ be the
 projection map. Denote the preimage of the midpoint of an edge $e\in \Gamma$
 under $p$ by $S_e$. For a vertex $v\in\Gamma$, we consider the
 component containing $v$ of $\Gamma$ cut open along the midpoints
 of edges. Let $X_v$ denote the preimage of this component under
 $p$, called \emph{vertex space}. For any vertex $v\in \Gamma$,
 the vertex surface $F_v$ is a deformation retract of the vertex
 space $X_v$.

For each edge $e$ of  $\Gamma$, the lift to the universal covers
of the finite index covering map $p_e: S_e\rightarrow F_{o(e)}$ is
a quasi-isometry. This is precisely the 'quasi-isometrically
embedded condition' in \cite{BestvinaFeighn:combination}. We may
omit this condition from the hypothesis of the combination theorem
for the cases of the graphs of surfaces.

Define a continuous function $\Delta : [-k,k]\times I\rightarrow
\widetilde{S\Gamma}$ to be a \emph{hallway} of length $2k$, if for
any $i$ from $-k$ to $k$, $\Delta(\{i\}\times I)$ is a geodesic in
$\widetilde S_{e(i)}$, and $\Delta ((i,i+1)\times I)$ stays in the
interior of $\widetilde X_{v(i)}$. Suppose $\Delta ([i,i+1]\times
I)$ stays in the closure of $\widetilde X_{v(i)}$. $\Delta$ is
$\rho$-$thin$ if $d_{\widetilde X_{v(i)}}(\Delta ((i,t)),\Delta
((i+1,t)))\leq \rho$ for $i\in \{-k,-k+1,\cdots,k-1\}$ and $t\in
I$. The hallway $\Delta$ is \emph{essential} if the edge path
$e(-k)\ast\cdots\ast
 e(k)$ never backtracks in the Bass-Serre tree  $t$, i.e., $e(i)\neq
\overline e(i+1)$ for $i\in \{-k,\cdots,k-1\}$. The $girth$ of
$\Delta$ is the length of $\Delta( \{0\}\times I)$. Let
$\lambda>1$. The hallway $\Delta$ is $\lambda$-$hyperbolic$ if
$\lambda l(\Delta(\{0\}\times I))\leq max\{l(\Delta(\{-k\}\times
I)),l(\Delta (\{k\}\times I))\}$. The graph of surfaces $S\Gamma$
is said to satisfy the \emph{hallways flare condition} if there
exist numbers $\lambda>1$ and $k\geq 1$ such that for any $\rho$
there exists a constant $H(\rho)$, such that any $\rho$-thin
essential hallway of length $2k$ and girth at least $H(\rho)$ is
$\lambda$-hyperbolic.

\begin{theorem} (Combination Theorem) Let $S\Gamma$ be a
 graph of surfaces. Suppose
that $S\Gamma$ satisfies the hallways flare condition, then the
fundamental group of  $S\Gamma$ is hyperbolic.

\end{theorem}

\emph {Remarks}: 1. Notice that in the Bestvina-Feighn's
combination theorem, the vertex spaces are used in defining the
hallways; but in the proof of Theorem 1, we use the vertex
surfaces instead. We are allowed to do so, because the vertex
surface is a retraction of the corresponding vertex space, and
their universal covers are quasi-isometric to each other.

2. The hallways in the combination theorem are 'edge hallways',
i.e., the rungs $\Delta(i)\times I$ of the hallway $\Delta$ are
geodesics in the edge surfaces. In Theorem 1, the hallways are
'vertex hallways', i.e., $\Delta(i)\times I$ are geodesics in the
vertex surfaces. Since the covering map from an edge surface to a
vertex space is a quasi-isometry, and the vertex spaces and the
vertex surfaces are quasi-isometric,  if the vertex hallways flare
condition is satisfied then the edge hallways flare condition is
satisfied.

3. For the cases studied in this paper, we will prove the hallways
flares condition for length 2 hallways only.

\bigskip

 \noindent\textbf{Construction of pseudo-Anosov homeomorphisms} (The materials is covered by
\cite{Penner:recipe})

 In a surface $S$, $\mathcal C$ is an \emph{essential curve system}, if $\mathcal
C=\{c_1,\ \cdots,\ c_n\}$, where $c_1,\ \cdots,\ c_n$ are
non-trivial simple closed curves on $S$ which are pairwise
disjoint and pairwise non-homotopy.

Let $\mathcal C$ and $\mathcal D$ be two disjoint  essential curve systems, $\mathcal C$ hits
$\mathcal D$ \emph{efficiently} if $\mathcal C$ intersect $\mathcal D$ transversely, and no component on
$S\backslash(\mathcal C \cup \mathcal D)$ is a \emph{bigon}, an
interior of a disc whose boundary consists of one arc of $C\in
\mathcal C$ and one arc of $D\in \mathcal D$. We say that
$\mathcal C \cup \mathcal D$ \emph{fills} $S$ if the components of
the complement of $(\mathcal C\cup \mathcal D)$ are disks.

The following shows how to construct pseudo-Anosov homeomorphisms.
\begin{theorem}(\cite{Penner:recipe})
Suppose that $\mathcal C$ and $\mathcal D$ are essential curve
systems in an oriented surface $F$ so that $\mathcal C$ hits
$\mathcal D$ efficiently and $\mathcal C\cup \mathcal D$ fills
$F$. Let $R(\mathcal C^{+}, \mathcal D^{-})$ be the free semigroup
generated by the Dehn twists $\{\tau_c^{+}:c\in \mathcal C \} \cup
\{\tau_d^{-1}:d\in \mathcal D\}$. Each component map of the
isotopy class of $\omega \in R(\mathcal C^{+},\mathcal D^{-})$ is
either the identity or pseudo-Anosov, and the isotopy class of
$\omega$ is itself pseudo-Anosov if each $\tau_c^{+}$ and
$\tau_d^{-1}$ occur at least once in $\omega$.\label{penner}
\end{theorem}

\bigskip

\noindent\textbf{Surface group extensions} (The materials is
covered by \cite{Mosher:Geosurvmcg} and
\cite{farbmosher:convexcocompmcg})

 A \emph{surface group
extension} is a short exact sequence of the form
\begin{equation}
\begin{CD}
 1\rightarrow\pi_1(S,x)\rightarrow \Gamma\rightarrow G\rightarrow 1
\end{CD}
\label{surf extension}
\end{equation}
where $S$ is a closed, oriented surface of genus $g\geq 2$. The
canonical example is the sequence
\begin{equation}
\begin{CD}
 1\rightarrow\pi_1(S,x)@>i>>MCG(S,x)@>q>>MCG(S)\rightarrow 1
\end{CD}
\label{canon surf exten}
\end{equation}
where $MCG(S)$ is the mapping class group of $S$, $MCG(S,x)$ is
the mapping class group of $S$ punctured at $x$. This short exact
sequence is universal for surface group extension, meaning that
for any extension as in (\ref{surf extension}), there exists a
commutative diagram
\begin{equation}
\begin{CD}
1@>>>\pi_1(S,x)@>>>\Gamma@>>> G@>>>1\\
@.@VVV@VVV@V\alpha VV\\
1@>>>\pi_1(S,x)@>i>>MCG(S,x)@>q>>MCG(S)@>>> 1
\end{CD}
\end{equation}
where $\Gamma$ is identified with the pushout group
\begin{equation}
\Gamma_{\alpha}=\{(\phi,\gamma)\in MCG(S,x)\times G|
q(\phi)=\alpha(\gamma)\},
\end{equation}
 $\alpha$ is a homomorphism from $G$ to $MCG(S)$, and the
homomorphisms $\Gamma\rightarrow G$ and $\Gamma\rightarrow
MCG(S,x)$ are the projection homomorphisms of the pushout group.
We are more interested in the case where $\alpha$ is an inclusion.

\bigskip

\noindent \textbf{Virtual centralizer of $\Phi$ } (The material is covered by  \cite{Mosher:Geosurvmcg})
 
 Given a subgroup $H$ of a group $G$, the \emph{virtual centralizer} $VC(H)$ of $H$
in $G$ is the subgroup of all $g\in G$ which commute with a finite
index subgroup of $H$. The virtual centralizer of an infinite
cyclic pseudo-Anosov subgroup has a nice geometric description.
Let $\mathcal {PML(S)}$ denote the space of projective measured
laminations of the surface $S$. Let $\Lambda^s,\ \Lambda^u\subset
\textbf{P} \mathcal {ML}$ be the fixed points of a pseudo-Anosov
mapping class $\Phi$, and let $Fix\{\Lambda^s,\Lambda^u\}$ denote
the subgroup in $MCG(S)$ whose elements fix $\Lambda^s,\
\Lambda^u$ point wise. \cite{Mosher:Geosurvmcg} shows that
$Fix\{\Lambda^s,\Lambda^u\}=VC\langle\Phi\rangle$.
\bigskip

\noindent\textbf{Facts of hyperbolic geometry} (The material is
covered by \cite{BridsonHaefliger} and
\cite{CassonBleiler:handwritten})

 Our proofs make heavy use of the following facts of hyperbolic space, $H^{2}$,
 geometry:

\emph{Fact 1}. For any $0<\delta<1$, and $D>0$, there exists $l(\delta,
D)$, such that if $\gamma, \alpha$ are geodesic segments of length
at least $l(\delta,D)$, and the end points $x,\ y$ of $\gamma$
have distance at most $D$ from the end points $x^{\prime},\
y^{\prime}$ of $\alpha$ respectively, then there exist subsegments
$\gamma^{\prime}\subset \gamma$, $\alpha^{\prime}\subset \alpha$
of lengths at least $(1-\delta)Length(\gamma)$ and
$(1-\delta)Length(\alpha)$ respectively, such that the Hausdorff
distance between $\gamma^{\prime}$ and $\alpha^{\prime}$ is less
than $\delta$.

Roughly speaking, for any two geodesic segments, if their end
points have bounded distances from each other, then most part of
them can be arbitrarily close to each other as long as the
segments are long enough.

\emph{Fact 2}. Given $k\geq 1,\ c\geq 0$, there exists a constant
$N_0(k,c)$, such that any $(k, c)$ quasi-geodesics line or segment
in $H^{2}$ has Hausdorff distance at most $N_{0}(k, c)$ from a
geodesic line or segment with the same end points.

\emph{Fact 3}. Let $\Lambda_1$ and $\Lambda_2$ be two minimal geodesic
laminations filling a hyperbolic surface $S$. If their lifts
$\widetilde \Lambda_1$ and $\widetilde \Lambda_2$ on the universal
cover of $\widetilde S$ have at least one end point in common,
then $\Lambda_1=\Lambda_2$. A geodesic lamination $\Lambda$ is
\emph{minimal} if every leaf $L$ is dense, that is, $\overline
L=\Lambda$. A geodesic lamination $\Lambda\subset S$ is a
\emph{filling} lamination if no simple closed curve in $S$ is
disjoint from $\Lambda$.

 The
reason is that two minimal filling surface geodesic laminations
either transversely intersect with each other or are equal to each
other.

From \cite{FLP}, we know that the stable and unstable geodesic
laminations of a pseudo-Anosov homeomorphism are minimal and
filling.

\bigskip

\section {Main Theorem}
We will give a new version of Mosher's parallel corresponds lemma
and use it to prove Theorem 1. Moreover we will reformulate the
hypothesis of Theorem 1. The original corresponds lemma of Mosher is in \cite{Mosher:hypbyhyp}.

\subsection {New version of the parallel corresponds lemma}

Consider a pseudo-Anosov mapping class $\Phi\subset MCG(S)$, let
$\phi\in Homeo(S)$ be a pseudo-Anosov representative with the
stable and unstable measured foliations $f_{\phi}^{s},\
f_{\phi}^{u}$. Recall that the transverse measures on
$f_{\phi}^{s}$ and $f_{\phi}^{u}$ define a singular Euclidean
structure on $S$, with isolated cone singularities. We call the
leaves of $f_{\phi}^{s}$  the horizontal leaves and the leaves of
$f_{\phi}^{u}$  the vertical leaves. The singular Euclidean
structure determines a metric $d_{\phi}$ on $S$ for which each
path can be homotopic to a unique geodesic rel end points. The
lifts to the universal covers of the hyperbolic metric and the
singular Euclidean metric are quasi-isometric equivalent.

In the following, for a homotopy class $\gamma$ of a curve rel. end points, let $\gamma^h$ denote
the hyperbolic geodesic segment in the homotopy class of $\gamma$,
and let $\gamma^E$ denote the singular Euclidean geodesic segment
in the same homotopy class. Let  $|\cdot|$ denote the hyperbolic
metric, and let $|\cdot|_E$ denote the singular Euclidean metric.
For a homotopy class $\gamma$, let $|\gamma|$ denote the
hyperbolic length of $\gamma^h$, let $|\gamma|_E$ denotes the
singular Euclidean length of $\gamma^E$.

Given  $0<\eta<1$, define $slope_{\phi}^{\eta}$ to be the set of
all homotopy classes $\gamma$, such that the (unsigned) Euclidean
angle between $\gamma^E$ and $f_{\phi}^{s}$ is at least $\eta$, on
a subset of $\gamma^E$ of length at least $\eta\cdot Length \
\gamma^E$. Given $\lambda>1$, let
$stretch_{\phi}^{\lambda}=\{\gamma\mid
|(\phi(\gamma))|>\lambda|\gamma|\}$. Let $n$ be a large enough
integer, such that if the vector $v\in E^2$ has an angle at least
$\eta$ with the horizontal axis, then the matrix
$$
\left(
\begin{array}{cc}
\lambda_{\phi}^{-n} & 0\\
0 & \lambda_{\phi}^{n}
\end{array}
\right)
$$
stretches $v$ by a factor of at least $\lambda/\eta$, where
$\lambda_{\phi}=\lim_{i\rightarrow \infty}|\phi^i(\alpha)|^{1/i}$ is the stretching factor of $\phi$, $\alpha$ is a simple closed geodesic on $S$. Since the
singular Euclidean metric is quasi-isometric to the hyperbolic
metric, it follows that given $\phi$, $0<\eta<1$, and $\lambda>1$,  there exists $N$
such that if $n\geq N$, then $slope_{\phi}^{\eta}\subset
stretch_{\phi^n}^{\lambda}$.

An $\eta$-$lever$ is a homotopy from a singular Euclidean geodesic
segment $\alpha$ to a horizontal segment $\beta$, where $\beta$ is
a segment of a nonsingular leaf of the horizontal foliation
$f_{\phi}^s$, such that each track of the homotopy is a vertical
geodesic segment, maybe degenerate, and each point of
int($\alpha$) is disjoint from singularities during the homotopy,
and int($\alpha$) makes an angle at most $\eta$ with the
horizontal leaves. In \cite{Mosher:hypbyhyp}, $\beta$ is not
necessary to be a segment of a nonsingular leaf. But we can always
make $\beta$ be a segment of a nonsingular leaf, because there
exist nonsingular leafs which are arbitrary close to a singular
leaf. Notice that the angle between a singular Euclidean geodesic
and the horizontal leaves changes only  when the singular
Euclidean geodesic passes a singularity. Therefore the interior of
$\alpha$ has a constant angle with the horizontal leaf.

A lever is denoted by $(\alpha,\beta)$, where $\alpha$ is called
the \emph{inclined edge} of the lever, and $\beta$ is called the
\emph{horizontal edge} of the lever. A lever is \emph{maximal} if
and only if a singularity contained in the track of each end point
of $\alpha$. The \emph{length} of the lever is $|\alpha|_E$,
the \emph{height} of the lever is the maximum length of the tracks
of the points of $\alpha$, which is achieved at the endpoints.

\begin{proposition}
For any $l$, $H>0$, there exists $\eta(l, H)>0$, so that every
maximal $\eta$-lever has length at least $l$ and height at most
$H$.
\end{proposition}

The proof is given in the first seven paragraphs of the proof of
the sublemma on page 3451 in  \cite{Mosher:hypbyhyp}. This
proposition will be used in the proof of the following lemma.

In the proof of the following lemma, we need some facts. It is
well known that the measured foliations $f_{\phi}^s$, $f_{\phi}^u$
can be straightened to measured geodesic laminations $l_{\phi}^s$,
$l_{\phi}^u$. Actually, there is a 1-1 correspondence between
leaves of $l_{\phi}^s$ and smooth leaves of $f_{\phi}^s$, where a
smooth leaf is either a nonsingular leaf or the union of two
singular half-leaves meeting at a singularity with angle $180^0$.
Similarly for $f_{\phi}^u$. The singularities are discrete, so the
length of any geodesic between them has a positive lower bound.

\begin{lemma}(New version of Parallel Corresponds lemma) Given any
pseudo-Anosov  homeomorphism $\phi$ and $0< \epsilon  <1$, there
exist $0<\eta<1$ and $L>0$ such that for any homotopy class
$\gamma$, if $\gamma\notin slope^{\eta}_{\phi}$ and
$|\gamma|_E\geq L$, then on a subset of $\gamma^h$ of length at
least $(1-\epsilon)Length(\gamma^h)$, the distance between the
tangent line of $\gamma^h$ and the set $\l^s_{\phi}$, measured in
$PS$, is at most $\epsilon$.
\end{lemma}

The differences between the Parallel Corresponds lemma in
\cite{Mosher:hypbyhyp} and this new version are as follows. In
\cite{Mosher:hypbyhyp}, the Parallel corresponds lemma works only
for closed based geodesics, and the word metric is used to define
the stretching factor; in this paper, the new version of the
parallel corresponds lemma works for non closed geodesics as well,
and the hyperbolic metric is used to define the stretching factor.

\begin{proof}

The first step is to find long subsegments $\alpha_i\subset
\gamma^E$ and segments $\beta_i$ of leaves of $f_{\phi}^s$, such
that $\alpha_i$ is homotopic to $\beta_i$ by homotoping through
short paths. Then we shall project $\alpha_i$ to a subsegment of
$\gamma^h$ and project $\beta_i$ to a segment of a leaf $B_i^h$ of
$l_{\phi}^s$, and show that a big portion of these projections are
very close to each other. Finally we shall prove most part of
$\gamma^h$ are covered by  big portion of these projections.

For  $\gamma\notin Slope_{\phi}^{\eta}$, let
$\{(\alpha_i,\beta_i)\}$ be the set of all maximal $\eta$-levers
of $\gamma^E$, where the inclined edge $\alpha_i$ is a subsegment
of $\gamma^E$ and the horizontal edge $\beta_i$ is a segment of
some non-singular leaf $B^E_i$ of $f_{\phi}^s$.

Step 1: first, let  $H=1$, by proposition 4, we know that for any
$l>0$, there exists $\eta>0$ such that every maximal $\eta$-lever
$\{(\alpha_i,\beta_i)\}$ has length at lest $l$ and height at most
$H=1$. The first step is proven.

Step 2: we shall construct long subsegments of $\gamma^h$ from the
inclined edges $\alpha_i\subset \gamma^E$ of the maximal levers,
such that these long subsegments of $\gamma^h$ have small distance
with $l_{\phi}^s$ measured in $PS$. In the rest of this lemma, the
distance and length mean hyperbolic distance and length, otherwise
we will use the notations Euclidean distance and length. I will
use the notation "E" to represent the Euclidean distance and
length.

We know that any non-singular leaf $B_i^E$ of $f_{\phi}^s$ is a
$k,c$ quasi-geodesic under the hyperbolic metric, and it can be
straightened to  a unique leaf $B_i^h$ of $l_{\phi}^s$. Let
$\delta_i\subset \gamma^h$ and $\sigma_i\subset B^h_i$ denote the
closest point projections from $\alpha_i\subset \gamma^E$ to
$\gamma^h$, and  from $\beta_i\subset B_i^E$ to $B^h_i$
respectively. We shall see that most portion of $\delta_i$ has
small distance with $\sigma_i$, for all $i$.

Since $\gamma^E$ is a $k,c$ quasi-geodesic segment contained in the
$N_0(k,c)$ neighborhood of $\gamma^h$, and $\delta_i$,
$\alpha_i$ are subsegments of $\gamma^h$, $\gamma^E$ respectively,
it follows that the distances between the end points of $\delta_i$
and $\alpha_i$ are not greater than $N_0$. For the same reason,
the distance between the end points of $\sigma_i$ and $\beta_i$
are not greater than $N_0$. The singular Euclidean distances
between the end points of $\beta_i$ and $\alpha_i$ is less than
the height $H=1$.  The  hyperbolic distances between the end point
of them are at most $mk$ for some $m>0$, because  the singular
Euclidean and hyperbolic metric are $k,c$ quasi-isometric to each
other. Therefore the distances between the end points of
$\delta_i$ and $\sigma_i$ are less than $2N_0+mk$. According to
Fact 1, for any $\epsilon_1>0$, there exists $L_1$ depending on
$2N_0+mk$ and $\epsilon_1$, if the length of $\delta_i$ is greater
than $L_1$, then more than $(1-\epsilon_1)|\delta_i|$ part of
$\delta_i$ has distance less than $\epsilon_1$ with $\sigma_i$.

The condition on the length of $\delta_i$ greater than $L_1$ is
easy to satisfy. Since $\alpha_i$ is a quasi-geodesic segment
whose end points have distances less than $N_0$ with the end
points of $\delta_i$, there exists $l_1>0$, such that if the
Euclidean length of $\alpha_i$ is greater that $l_1$, then the
length of $\delta_i$ is greater than $L_1$. By applying the step 1, we may now choose
$\eta$ small enough, so that the Euclidean length of $\alpha_i$ is
greater than $l_1$ for any $i$. Therefore more than
$(1-\epsilon_1)|\delta_i|$ part of $\delta_i$ has distance less
than $\epsilon_1$ with $\sigma_i$.

So far, we have proved that for any $\epsilon_1$, there exists
$\eta$, such that if $\gamma\notin slope_{\phi}^{\eta}$, then we
can locate long subsegments $\delta_i$ of $\gamma^h$, such that
more than $(1-\epsilon_1)$ of the length of $\delta_i$ has
distance less than $\epsilon_1$ with $\sigma_i\subset B_i^h$, for
any $i$.

Step 3: we will prove that $(1-\epsilon_1)\sum_i|\delta_i|$ part
of $\cup_i(\delta_i)$ covers most part of $\gamma^h$. We call this
$(1-\epsilon_1)\sum_i|\delta_i|$ part of $\cup_i(\delta_i)$ the
'good' part of $\gamma^h$.

Since  $\gamma\notin slope^{\eta}_{\phi}$, on a subset of
$\gamma^E$ of length at least $(1-\eta)|\gamma|_E$, the angle
between $\gamma^E$ and $f_{\phi}^s$ is less than $\eta$, i.e., the
$\eta$-levers cover more than $(1-\eta)$ part of $\gamma^E$. The
worst situation is that the two end subsegments of $\gamma^E$ are
covered by $\eta$-levers with lengths less than $l_1$. In this
case, after straightening, the end subsegments of $\gamma^h$ may
not have distances less than $\epsilon_1$ with $B^h$. We will only
prove this lemma for the worst situation, i.e., more than
$(1-\epsilon_1)|\gamma|_E$ part of $\gamma^E$ is covered by the
union of the maximal $\eta$-levers $(\alpha_i,\beta_i)$ and two
end $\eta$-levers which cover the two end segments of $\gamma^E$
respectively and with lengths less than $l_1$.

In the following the quasi-isometries will be replaced by
bi-Lipschitz maps when dealing with long segments. In the rest of
this proof, let $|\alpha_i|$ denote the length of the hyperbolic
geodesic which is homotopic to $\alpha_i$ rel. end points, and let
$|\alpha_i|_E$ denote the Euclidean length of the singular
Euclidean geodesic $\alpha_i$. Keep in mind that none of the
following $\delta_i$ is the projection of an end subsegment of
$\gamma^E$.

\begin{align*}
(1-\epsilon_1)\sum_i|\delta_i|&\geq
(1-\epsilon_1)(\sum_i(|\alpha_i|-2N_0))\\
&\geq(1-\epsilon_1)(\sum_i(\frac{|\alpha_i|_E}{k}-2N_0))\\
\intertext{According to Proposition 4, we can take $\eta$ to be
small enough, so that $|\alpha_i|_E\geq l_2=4kN_0$ for any $i$}
&\geq (1-\epsilon_1)\frac{\sum_i|\alpha_i|_E}{2k}\\
\intertext{Since the union of the $\eta$-levers-the maximal
$\eta$-levers and the two end $\eta$-levers, covers more than
$(1-\epsilon_1)|\gamma|_E$ part of $\gamma^E$, and we suppose that
the two end $\eta$-levers have lengths less than $l_1$,}
&\geq (1-\epsilon_1)\frac{(1-\epsilon_1)|\gamma|_E-2l_1}{2k}\\
\intertext{Take $|\gamma|_E$ to be long enough, so that
$|\gamma|_E\geq L_2=\frac{2l_1}{\epsilon_1}$}
&\geq (1-\epsilon_1)\frac{(1-2\epsilon_1)|\gamma|_E}{2k}\\
&\geq \frac{(1-2\epsilon_1)^2|\gamma|_E}{2k}\\
\end{align*}

Hence, $(1-\epsilon_1)\sum_i|\delta_i|\geq
\frac{(1-2\epsilon_1)^2|\gamma|_E}{2k}$.

The `bad' parts of $\gamma^h$ are of three kinds. The first kind
of bad part is the two end subsegments of $\gamma^h$ which have
lengths less than $L_1$. The sum of the lengths of  the end
subsegments of $\gamma^h$ is at most $2L_1$. We can take
$|\gamma|_E$ to be big enough such that $2L_1\leq
\epsilon_1|\gamma|_E$.

 The second kind of bad part of $\gamma^h$ is the
$\epsilon_1|\delta_i|$ part of $\delta_i$'s which may be out of
the $\epsilon_1$ neighborhood of $\sigma_i$. Since the projection
map can not prolong length, and the distances between the ends of
$\alpha_i$ and $\delta_i$ are not greater than $N_0$,

\begin{align*}
\sum_i\epsilon_1|\delta_i|&\leq
\epsilon_1\sum_i(|\alpha_i|+2N_0)\\
 \intertext{We can take $\eta$ to be small enough, so that $|\alpha_i|_E\geq l_2=4kN_0$ for any $i$.
The singular Euclidean metric and the hyperbolic metric are $k$
bi-Lipschitz shows that $|\alpha_i|_E\leq k|\alpha_i|$. Therefore
$2N_0\leq 2kN_0\leq \frac{|\alpha_i|}{2}$}
&\leq \epsilon_1\frac{3}{2}\sum_i|\alpha_i|\\
&\leq \epsilon_1\frac{3}{2}k\sum_i|\alpha_i|_E\\
&\leq \epsilon_12k|\gamma|_E\\
\end{align*}

The third kind of bad part of $\gamma^h$ are the projections of
$\epsilon_1|\gamma|_E$ part of $\gamma^E$ which has slope greater
than $\epsilon_1$ with $f_{\phi}^s$. Let $\xi_i$ denote this kind
of subsegment of $\gamma^E$. There is a lower bound $b$ of the
Euclidean lengths of $\xi_i$ for all $i$, which equals the minimum
of the Euclidean distances between singularities. The sum of the
lengths of the projections from $\xi_i$ to $\gamma^h$ is at most
$\sum_i(k|\xi_i|_E+c)\leq \sum_i(k|\xi_i|_E+(n-1)kb)\leq
n\sum_i(k|\xi_i|_E)\leq nk\epsilon_1|\gamma|_E$, for some $n$
satisfies $c\leq (n-1)kb$.

Therefore, the length of the `bad' part of $\gamma^h$ is at most
the sum of the above three kinds, which is
$(2k+1+nk)\epsilon_1|\gamma|_E$. Hence the ratio of the `good'
part of $\gamma^h$ to the `bad' part of $\gamma^h$ is at least
$\frac{(1-2\epsilon_1)^2}{2k(1+2k+nk)\epsilon_1}$. It is easy to
see, for any constant $\epsilon$ there exists a small enough
$\epsilon_1$, such that the ratio of the `good' part of $\gamma^h$
to $\gamma^h$ is at least $(1-\epsilon)$.

To recap: for any $\epsilon>0$, we can choose small enough
$\epsilon_1$, so that
$\frac{(1-2\epsilon_1)^2}{2k(1+2k+nk)\epsilon_1}$ is greater than
$1-\epsilon$, therefore the 'good' part of $\gamma^h$ covers more
than $(1-\epsilon)$ of the total length of $\gamma^h$. Then choose
$\eta$ small enough so that if $\gamma^E\notin
slope_{\phi}^{\eta}$, then more than $(1-\epsilon)|\delta_i|$ part
of $\delta_i$ has distance less than $\epsilon_1$ with $\sigma_i$.
In addition, take $|\gamma|_E$ to be at least $L$, where
$L=max\{L_2,2L_1/\epsilon_1\}$. Hence if $\eta$ is small enough,
$\gamma\notin slope_{\phi}^{\eta}$ and $|\gamma|_E\geq L$, then
most part of $\gamma^h$ has distance at most $\epsilon$ to
$l_{\phi}^s$, measured in $PS$.

\end{proof}
Given a geodesic lamination $\Lambda$ and  $0<\epsilon<1$, let
$WN_{\epsilon}(\Lambda)$ denote the set of all the homotopy class
$\gamma$, so that on a subset of $\gamma^h$ of length at least
$(1-\epsilon)Length(\gamma^h)$, the distance from the tangent line
of $\gamma^h$ to the set $\Lambda$, measured in $PS$, is at most
$\epsilon$. Using this notation, the parallel corresponds lemma
says that for any $0<\epsilon<1$, there exists $0<\eta<1$ and
$L>0$, such that if $\gamma\notin slope_{\phi}^{\eta}$ and
$|\gamma|_E\geq L$, then $\gamma\in WN_{\epsilon}(\Lambda^s)$, where $\Lambda^s$ is the measured stable geodesic lamination of $\phi$.

\subsection {Proof of the main theorem}

\begin{proof}[Proof of Theorem 1] We shall prove that there exist $\lambda>1$ and $C>0$,
so that for any vertex $w\in \Gamma$, if a  based geodesic segment
$\gamma^h_w\subset F_w$ has length at least $C$, then all but at
most one preimages of it are stretched by corresponding
$\phi_i^{m_i}$ by a factor of at least $\lambda$, for any $i\in
I_w$, where $I_w=\{i|\ e_i \text{ is an oriented edge such that
the origin}$
 $\text{of } e_i \text{ is } w\}$. Hence the hallways flare
condition is satisfied. Therefore
$S\Gamma_{\boldsymbol{\varphi^m}}$ is a hyperbolic surface.

Let $v$ be a vertex of $\Gamma$, let $\gamma_v^h\subset F_v$ be a
based geodesic segment. Consider the set $\Sigma=\bigcup_{i\in
I_v} p^{-1}_i(\gamma^h_v)$, where $p^{-1}_i(\gamma^h_v)$ is the
set of all preimages of $\gamma^h_v$ under the map $p_i$. Notice
that all the elements of $\Sigma$ are based geodesics, since the
edge surfaces of $S\Gamma_{\boldsymbol{\varphi^m}}$ equipped with
the pullback metrics.

First, we claim that there exist $0<\epsilon_0<1$ and $H_0>0$,
such that if the length of $\gamma^h_v$ is greater than $H_0$,
then at most one of the elements of $\Sigma$,  say $\beta\in
p^{-1}_{i_0}(\gamma_v^h)$, such that $\beta\in
WN_{\epsilon_0}(\Lambda_{i_0}^s)$, for some $i_0\in I_{v}$; all
other elements of $\Sigma$ are not contained in
$WN_{\epsilon}(\Lambda_{i}^s)$ for corresponding $\Lambda^s_i$.
Second, according to Lemma 5, for this $\epsilon_0$, there exist
$0<\eta(\epsilon_0)<1$ and $L(\epsilon_0)>0$, such that  any
$\alpha\in \Sigma$ with length $|\alpha|=|\gamma_v^h|$ greater
than $L(\epsilon_0)$, if $\alpha\notin
WN_{\epsilon_0}(\Lambda^s_j)$, then $\alpha\in
slope_{\phi_j}^{\eta(\epsilon_0)}$. Therefore $\alpha$ is
stretched by $\phi_j^{m_j}$ by a factor of at least $\lambda$ for
sufficiently large $m_j$. Combining these, we know that for any
$\gamma_v^h$ with length greater $C=\max \{H, L(\epsilon_0)\}$,
all but at most one preimages of $\gamma_v^h$ are stretched by
corresponding $\phi_i^{m_i}$ by at least a factor $\lambda$.

Suppose the claim is not true. Namely for any
$\epsilon_n\rightarrow 0$, and any $H_n\rightarrow \infty$, there
exist based geodesic segments $\gamma_n^h\subset F_v$ with lengths
at least $H_n$, by passing to a subsequence, without loss of
generality,  suppose $A_n^h\in p_1^{-1}(\gamma^h_n)$ and $B_n^h\in
p_2^{-1}(\gamma^h_n)$, such that $A_n^h\in
WN_{\epsilon_n}(\Lambda^s_1)$ and $B_n^h\in
WN_{\epsilon_n}(\Lambda^s_2)$. Project $A_n^h$ and $B_n^h$ to $\Lambda^s_1$ and $\Lambda^s_2$ respectively, there exist long
subsegments $\nu_n\subset \Lambda_1^s$ and $\omega_n\subset
\Lambda^s_2$, such that $|\nu_n|$, $|\omega_n|\rightarrow \infty$,
and the distance between $Dp_1|T\nu_n$ and $Dp_2|T\omega_n$
converges to zero. This conflicts with the fact that
$Dp_1|T\Lambda^s_1$ and $Dp_2|T\Lambda^s_2$ are disjoint.

\end{proof}

\bigskip
\subsection {Reformulation of Theorem 1}

Notations here are the same as in the introduction. The only
difference is the edge surfaces are not necessary equipped
with the pullback metrics here.

Let $v$ be a vertex of $\Gamma$,  let $y$ be the base point of
$F_v$, and let $I_v$ be as defined before. Consider the set
$p_{i}^{-1}(y)\subset S_{i}$ of all the points of $S_i$ that cover
$y$ via the map $p_i$, for $i\in I_v$. Denote $X=\cup _{i\in
I_v}p^{-1}_i(y)$.

 Suppose $a\in p_i^{-1}(y)$, choose
a lift $\widetilde p_a:(\widetilde S_{i},\widetilde a)\rightarrow
(\widetilde F_v,\widetilde y)$, where $\widetilde S_{i}$ and $\widetilde F_v$ are the universal covers of $S_i$ and $F_v$ respectively. Let $\Lambda^s_i\subset S_{i}$ be
the stable lamination of $\phi_{i}$, and let $\widetilde \Lambda
^s_{i}\subset \widetilde S_{i}$ be the lift of $\Lambda^s_{i}$.
 Notice that $\partial \widetilde p_a(
\widetilde \Lambda^s_{i})\subset
\partial \widetilde F_v$ is
well defined independent of the choice of $\widetilde y,\widetilde
a$. If for any $a\neq b\in X$, $\partial \widetilde p_a(\widetilde
\Lambda_i^s)\cap \partial \widetilde p_b(\widetilde
\Lambda_j^s)=\varnothing$, where $a\in p_i^{-1}(y)$, $b\in p_j^{-1}(y)$, then we say $v$
satisfies the \emph{disjointness condition}. We only ask $a\neq b$, but $i$ may equal to $j$. The reformulation of
Theorem \ref{main} is the following.
\begin{theorem}
 Let $S\Gamma_{\boldsymbol{\varphi^m}}$ be a finite
graph of surfaces with underlying graph $\Gamma$. If for any
vertex $v\in \Gamma$, the disjointness condition is satisfied,
 then $\pi_1(S\Gamma_{\boldsymbol{\varphi^m}})$
is a hyperbolic group, when $m_i\in \boldsymbol{m}$ are
sufficiently large. \label{reform}
\end{theorem}
\bigskip

We shall show the equivalence of the hypothesis of Theorem 1 and
Theorem 6.

First, suppose $Dp_i(T\Lambda^s_i)$ is disjoint from
$Dp_j(T\Lambda^s_j)$, for $i\neq j$. Then the images of the leaves $\Lambda^s_i$
under the map $p_i$ must transversely intersect the images of the
leaves $\Lambda^s_j$ under the map $p_j$. Thus the end points of
their lifts in $\widetilde F$ are disjoint.

Second, suppose $Dp_i(T\Lambda^s_i)$ is injection for all $i$. If
$\partial \widetilde p_{a_1}(\widetilde \Lambda_i^s)\cap
\partial \widetilde p_{a_2}(\widetilde \Lambda_i^s)\neq \varnothing$, for some
$a_1, a_2\in p^{-1}_i(y)$,  then there exist leaves $\widetilde
L_1, \widetilde L_2\subset \widetilde \Lambda^s_i$, such that
$\widetilde p_{a_1}(\widetilde L_1)=\widetilde p_{a_2}(\widetilde
L_2)$. It contradicts with the injectiveness of
$Dp_i(T\Lambda_i^s)$. We have finished the proof of one direction.

Suppose $Dp_i(T\Lambda_i^s)$ is not disjoint with
$Dp_j(T\Lambda_j^s)$, i.e., there exist leaves $L\subset
\Lambda_i^s$ and $J\subset \Lambda_j^s$, such that
$Dp_i(L)=Dp_j(J)$. Therefore there exist a lift $\widetilde L$  of
$L$, a lift $\widetilde J$  of $J$, such that $\widetilde
p_a(\widetilde L)=\widetilde p_b(\widetilde J)$ for some $a\in
p^{-1}_i(y)$ and some $b\in p^{-1}_j(y)$. It conflicts with the
hypothesis of Theorem 6. Similar proof for the injections of
$Dp_i(T\Lambda_i^s)$ for all $i$.

\bigskip

\section{Applications}

The theorem below will be used to prove Corollary \ref{finite
group}.
\begin{theorem}
(Farb \& Mosher \cite{FarbMosher:quasiconvex},Theorem 1.2) Let
$\pi_1(S)$ be the fundamental group of a surface $S$, and let
$\Gamma_{\alpha}$ be the surface group extension of a group $G$.
If $\Gamma_{\alpha}$ is word hyperbolic then the homomorphism
$\alpha: G \rightarrow MCG$ has finite kernel and convex cocompact
image. \label{con cocom sub}
\end{theorem}

\bigskip

\begin{corollary}
Let $G,\ H$ be finite subgroups of $MCG(S)$, and let $\Phi\in
MCG(S)$ be a pseudo-Anosov mapping class. If the virtual
centralizer of $\langle\Phi\rangle$ has trivial intersection with
$G$ and $H$, then  $\langle G, \Phi^M H \Phi^{-M}\rangle$ is a
free product in $MCG(S)$, i.e., $\langle G, \Phi^M H
\Phi^{-M}\rangle\cong G* \Phi^MH\Phi^{-M}$, and its extension
group is hyperbolic, for sufficiently large $M$.\label{finite
group}
\end{corollary}

Remark: if $G$ is a finite subgroup of $MCG(S)$, then $G$ has a
faithful representation, still called $G\subset Homeo(S)$. The
quotient $S/G$, called $F_0$, is a hyperbolic surface or orbifold.
There exists a canonical embedding $i:\mathcal
{PML}(F_0)\hookrightarrow \mathcal {PML}(S)$, where $\mathcal
{PML}$ is the projective measured geodesic laminations space.
Given a pseudo-Anosov mapping class $\Phi\in MCG(S)$, if the
stable and unstable geodesic laminations
$\Lambda^s,\Lambda^u\notin i(\mathcal {PML}(F_0))$, then the
virtual centralizer of $\langle\Phi\rangle$ has trivial
intersection with $G$. Therefore, it is very easy to find
$\Phi\subset MCG(S)$ which satisfies the hypothesis of this
corollary.

\begin{proof} Let the symbols $G, \ H$ denote both the finite groups of $MCG(S)$ and their faithful
representations in $Homeo(S)$. Let $F_0=S/G, \ F_1=S/H$. Let $p:
S\rightarrow F_0,\ q:S\rightarrow F_1$ denote the corresponding
covering maps, and let $p_{*}:\pi_1(S)\rightarrow \pi_1(F_0)$,
$q_{*}:\pi_1(S)\rightarrow \pi_1(F_1)$ denote the induced maps on
fundamental groups.

Let $G\Gamma$ be the graph of groups:
\begin{equation*}
\begin{CD}
\pi_1(F_0)@<p_*<<\pi_1(S)@>\Phi^M>>\pi_1(S)@>q_*>>\pi_1(F_1)
\end{CD}
\end{equation*}

$\pi_1(G\Gamma)$ is the fundamental group of the graph of surfaces $S\Gamma$:

\begin{equation*}
\begin{CD}
F_0@<p<<S@>\phi^M>>S@>q>>F_1
\end{CD}
\end{equation*}
where $\phi\in Homeo(S)$ is a pseudo-Anosov representative
homeomorphism of $\Phi$.

There exists a short exact sequence

\begin{equation*}
\begin{CD}
1\rightarrow \pi_1(S,x)\rightarrow
\Gamma_{G*\Phi^MH\Phi^{-M}}\rightarrow
G*\Phi^MH\Phi^{-M}\rightarrow 1
\end{CD}
\end{equation*}

It is not hard to see that $\Gamma_{G*\Phi^MH\Phi^{-M}}$ is
isomorphic to $\Gamma_G*_{\pi_1(S)} \Gamma_{\Phi^MH\Phi^{-M}}$,
and $\Gamma_G*_{\pi_1(S)} \Gamma_{\Phi^MH\Phi^{-M}}$ is isomorphic
to $\pi_1(G\Gamma)$.

 According to Theorem 7, if $\pi_1(G\Gamma)$ is a word hyperbolic group, then
 $\delta:G*\Phi^MH\Phi^{-M}\rightarrow MCG(S)$ has finite
kernel. Since $G$ and $\Phi^MH\Phi^{-M}$ are finite groups, by
applying Theorem 3.11 of Scott and Wall \cite{ScottWall}, a normal
subgroup of $G*\Phi^MH\Phi^{-M}$ must be trivial or finite index.
Therefore $\delta$ is an injection, which tells us that
$\langle G,\Phi^MH\Phi^M \rangle\cong G*\Phi^MH\Phi^{-M}$.

In order to prove $\pi_1(G\Gamma)$ is word hyperbolic, we only
need to show that $S\Gamma$ is a hyperbolic graph of surfaces.

Let $y\in F_0$ be the base point, let
$\{x_1,\cdots,x_r\}=p^{-1}(y)$ denote the preimages of $y$ under
the covering map $p$, and let $\widetilde x_i\in \widetilde S $ be
a covering point of $x_i$ for $i\in \{1,\cdots,r\}$. Let
$\widetilde p_i: (\widetilde S, \widetilde x_i)\rightarrow
(\widetilde F_0,\widetilde y)$ be a lift of $p$, let
$D_{ik}:(S,x_i)\rightarrow (S,x_k)$ be a deck transformation of
covering map $p$, and let $\widetilde D_{ik}:(\widetilde
S,\widetilde x_i)\rightarrow (\widetilde S,\widetilde x_k)$ be a
lift of $D_{ik}$.

According to Theorem \ref{reform}, if $\partial \widetilde
p_i(\widetilde \Lambda^s)\subset \partial\widetilde F_0$ are
pairwise disjoint on $\partial \widetilde F_0$, and the similar
condition holds on $\partial \widetilde F_1$, then $S\Gamma$ is
hyperbolic.

In the following, we only prove that $\partial \widetilde
p_1(\widetilde \Lambda^s)$ and $\partial \widetilde p_2(\widetilde
\Lambda^s)$ are disjoint; a similar argument holds for the
pairwise disjointness of $\{\partial\widetilde p_i(\widetilde
\Lambda^s)\}$ for all $i\in \{1,\cdots,r\}$, and the pairwise
disjointness of $\{\partial\widetilde q_j(\widetilde \Lambda^u)\}$
for all $j$.

Since $\widetilde p_1=\widetilde p_2\widetilde D_{12}$,
$\widetilde p_1(\widetilde \Lambda^s)=\widetilde p_2\widetilde
D_{12}(\widetilde \Lambda^s)$. Hence if the boundary points of the
images of $\widetilde \Lambda^s$ under $\widetilde p_1$ and
$\widetilde p_2$ have one point in common, then $\widetilde
D_{12}(\widetilde \Lambda^s)$ and $\widetilde \Lambda^s$ have one
end point in common. Since $\widetilde D_{12}(\widetilde
\Lambda^s)$ and $\widetilde \Lambda^s$ are the lift of the geodesic laminations
$D_{12}(\Lambda^s)$ and $\Lambda^s$ respectively, by Fact 3, we know
$D_{12}(\Lambda^s)=\Lambda^s$, where $D_{12}$ considered as an
element of $G\subset MCG(S)$. Applying Theorem 3.5 in
\cite{Mosher:Geosurvmcg}, if $D_{12}(\Lambda^s)=\Lambda^s$, then
$D_{12}$ is contained in the virtual centralizer of
$\langle\Phi\rangle$. This contradicts with the hypothesis that
the virtual centralizer of $\langle\Phi\rangle$ has trivial
intersection with $G$.
\end{proof}

\bigskip

Let $\mathcal G_{\phi^m}$ as in Figure 3, where $S$, $F$ are genus
3 and 2 tori. Let $p: S\rightarrow F$ and $q: S\rightarrow F$ be
covering maps, and let $\phi$ be a pseudo-Anosov homeomorphism of the
mapping class $\Phi$. Abusing of notations, we use $D_p$, $D_q$
for both the deck transformations of $p$, $q$ and the mapping classes of the
deck transformations. It is easy to see that the
deck transformation group $GD_p$ of $p$ contains only two
elements, $D_p$ and the identity, the same is true for the deck
transformation group of $q$. Abusing of notations, we let $GD_p$
denote both the deck transformation group of $p$ and its image in $MCG(S)$.

\begin{corollary}
Suppose $a:S^1\rightarrow F$ and $c:S^1\rightarrow S$ are simple
closed curves such that $p^{-1}(a(S^1))=c(S^1),\ c(S^1)\subset
q^{-1}(a(S^1))$, and $q^{-1}(a(S^1))$ is disconnected, as in
Figure 4. In addition, suppose the virtual centralizer of $\langle
\Phi \rangle$ has trivial intersection with the images of the deck
transformation groups of $p$ and $q$ in $MCG(S)$. Then
$\pi_1(\mathcal G_{\phi^m})$ is a hyperbolic group, when $m$ is
sufficiently large. \label {hnn}
\end{corollary}

\begin{proof} Let $z$ be the base point of $F$, let $x_1,x_2$ be the covering
points of $z$ through the covering map $p$, and let $y_1,y_2$ be
the covering points of $z$ through the covering map $q$ . Let
$\widetilde p_1:(\widetilde S,\widetilde x_1)\rightarrow
(\widetilde F,\widetilde z)$ and $\widetilde p_2:(\widetilde
S,\widetilde x_2)\rightarrow (\widetilde F,\widetilde z)$ be the
lifts of $p$, and let $\widetilde D_p:(\widetilde S,\widetilde
x_1)\rightarrow (\widetilde S,\widetilde x_2) $ be the lift of
$D_p$. Similar notations hold for $q$.

According to Theorem \ref{reform}, we only need to show that
$\{\partial \widetilde p_1(\widetilde \Lambda^s),\partial
\widetilde p_2(\widetilde \Lambda^s),\partial \widetilde
q_1(\widetilde \Lambda^u),\ \partial \widetilde q_2(\widetilde
\Lambda^u)\}$ is a pairwise disjoint set.

First, we shall prove that $\partial \widetilde p_1(\widetilde
\Lambda^s)\cap\partial \widetilde p_2(\widetilde
\Lambda^s)=\varnothing$, $\partial \widetilde q_1(\widetilde
\Lambda^u)\cap\partial \widetilde q_2(\widetilde
\Lambda^u)=\varnothing$.

We know that $\widetilde p_1(\widetilde \Lambda^s)=\widetilde
p_2\widetilde D_p(\widetilde \Lambda^s)$. If $\partial \widetilde
p_1(\widetilde \Lambda^s)$ and $\partial \widetilde p_2(\widetilde
\Lambda^s)$ are not disjoint, then $\widetilde
\Lambda^s=\widetilde D_p(\widetilde \Lambda^s)$, as discussed in
Corollary 8. It conflicts with the hypothesis that the virtual
centralizer of $\langle\Phi\rangle$ has trivial intersection with
$GD_p$ and $GD_q$.

Therefore $\partial \widetilde p_1(\widetilde \Lambda^s)$ and
$\partial \widetilde p_2(\widetilde \Lambda^s)$ are disjoint, the
same holds for $\partial \widetilde q_1(\widetilde \Lambda^u)$
and $\partial \widetilde q_2(\widetilde \Lambda^u)$.

Second, we claim that if there exist
$\partial \widetilde p_r(\widetilde\Lambda ^s)$ and $\partial
\widetilde q_t(\widetilde\Lambda ^u)$ are not disjoint, for some $r,t\in\{1,\ 2\}$,then
$p(\Lambda^s)=q(\Lambda^u)$ is a geodesic lamination on $F$. It
follows that $\Lambda^s$ is a fixed point of $GD_p\subset MCG(S)$.
Therefore the virtual centralizer of $\langle\Phi\rangle$ and the
deck transformation group have non-trivial intersection. A
contradiction.

In the following, we will prove the above claim.

Since $p_\ast(\pi_1(S))\neq q_\ast(\pi_1(S))$, and they are both
index two subgroups of $\pi_1(F)$,  $p_\ast(\pi_1(S))\cap
q_\ast(\pi_1(S))$ is an index 4 subgroup of $\pi_1(S)$. By
calculating the Euler characteristic, we know there is a genus
5 surface $G$, and covering maps $i$ and $j$, such that the
diagram below commutes, i.e,\begin{equation} pi=qj
\end{equation}

\begin{picture}(300,100)
\put(55,80){$G$} \put(55,75){\vector(-1,-1){30}}
\put(65,75){\vector(1,-1){30}} \put(35,65){i} \put(85,65){j}
\put(15,35){$S$} \put(95,35){$S$} \put(25,30){\vector(1,-1){30}}
\put(95,30){\vector(-1,-1){30}} \put(35,10){p} \put(85,10){q}
\put(55,-5){$F$} \put(130,35)\
\end{picture}

After straightening, the preimages of $i^{-1}(\Lambda^s)$ and
$j^{-1}(\Lambda^u)$ are geodesic laminations, called $\mathcal
L^s$ and $\mathcal L^u$, on $G$.

Without loss of generality, suppose $\widetilde p_1(\widetilde
\Lambda^s)$ and $\widetilde q_1(\widetilde \Lambda^u)$ have one
end point in common, then $\widetilde p_1\widetilde i(\widetilde
{\mathcal L}^s)$ and $\widetilde q_1\widetilde j(\widetilde
{\mathcal L}^u)$ have one end point in common. It follows that
$\widetilde {\mathcal {L}}^s$ and $\widetilde {\mathcal {L}}^u$
have one common end point. We claim that $\mathcal L^s$ and
$\mathcal L^u$ are minimal geodesic laminations and fill the
surface $G$. Therefore if they have one common end point in the
universal cover of $G$, then $\mathcal L^s=\mathcal L^u$. It is
not hard to see that $\mathcal L^s$ is connected and without
isolated leaves, thereby $\mathcal L^s$ is minimal according to
Corollary 4.7.2 in \cite{BleilerCasson}. $\mathcal L^s$ and
$\mathcal L^u$ fill $G$ because they are lifts of filling
laminations $\Lambda^s$ and $\Lambda^u$.

There exists some $m$, such that $\phi^m: S\rightarrow S$ is
lifted by $i$ and $j$ to homeomorphisms of $G$ respectively.
Denote the lift of $\phi^m:S\rightarrow S$ through $i$ as $\zeta:
G\rightarrow G$, and the lift of $\phi^{-m}$ through $j$ as
$\sigma: G\rightarrow G$. Notice that $\mathcal L^s$ is the stable
geodesic lamination of $\zeta$, $\mathcal L^u$ is the stable
geodesic lamination of $\sigma$. Since $\mathcal L^s=\mathcal
L^u$, there exist positive integers $k_1,\ k_2$, such that
$\zeta^{k_1}$ is homotopic to $\sigma^{k_2}$.

Since $\zeta^{k_1}$ is homotopic to $\sigma^{k_2}$ and $pi=qj$, we know: $pi\zeta^{k_1}$ is homotopic to
$qj\sigma^{k_2}$. $p\phi^{k_1m}i$ is homotopic to $q\phi^{-k_2m}j$,
because $\phi^{k_1m}i=i\zeta^{k_1}$ and
$\phi^{-k_2m}j=j\sigma^{k_1}$.

$p(c)$ is a notation for the closed curve
$p(c):S^{1}\rightarrow F$ which is the composition of
$c:S^1\rightarrow S$ with the covering map $p:S\rightarrow F$.
Similar notations are used for other compositions of closed curves
with covering maps. ${c^2}:{S^1}\rightarrow S$ is defined to be
the composition of the map $z\rightarrow z^2$ on the unit circle
$S^1$ with map $c:S^1\rightarrow S$. Let $[a]$, $[c]$ denote the
conjugacy classes in the fundamental group of $F$ which
represented by the simple closed curve $a$, $c$.

 Since $p(c)$ is
homotopic to $a^2$ and $q(c)$ is homotopic to $a$, it tells us
that $[a]\notin p_{\ast}(\pi_1(S))$, $[a]\in q_{\ast}(\pi_1(S))$,
and $[a]^2\in p_\ast(\pi_1(S))\cap q_\ast(\pi_1(S))$. Hence there
exists $\gamma:S^1\rightarrow G$ which is homotopic to a simple
closed curve, such that $i(\gamma)$ is homotopic to $c$ and
$j(\gamma)$ is homotopic to $c^2$. Therefore $p\phi^{k_1m}(c)$,
$p\phi^{k_1m}i(\gamma)$, $q\phi^{-k_2m}j(\gamma)$ and
$q\phi^{-k_2m}(c^2)$ are homotopic to each other.

 We claim that $q\phi^{-k_2m}(c)$ is homotopic to a simple closed curve on $F$.
Let $\beta$ be the closed geodesic on $F$ which is homotopic to
$q\phi^{-k_2m}(c)$. If $\beta$ is not simple, then there exits a
point $z\in \beta(S^1)$, and a simple closed curve
$\alpha:S^1\rightarrow S$ which is homotopic to $\phi^{-k_2m}(c)$,
such that $q(\alpha)=\beta$, and there exit two points $x_1\neq
x_2\in \alpha(S^1)$ such that $q(x_1)=q(x_2)=z$. Since
$p\phi^{k_1m}(c)$ is homotopic to $q\phi^{-k_2m}(c^2)$, there
exists a simple closed curve $\eta$ from $S^1$ to $S$ which is
homotopic to $\phi^{k_1m}(c)$, and whose image under the map $p$
goes around $\beta$ twice, to be more precise, $p(\eta)=\beta ^2$.
It follows that there are four different points
$y_1,y_2,y_3,y_4\in\eta(S^1) $ such that
$p(y_1)=p(y_2)=p(y_3)=p(y_4)=z$, which conflicts with the fact
that $p:S\rightarrow F$ is an index 2 covering map.

By iterating, we have:

$$pi\zeta^{nk_1} \ \text {is homotopic to} \ qj\sigma^{nk_2},\ \text{for all}\ n\in N $$
$$ p\phi^{nk_1m}i(\gamma)\ \text {is homotopic to} \ q\phi^{-nk_2m}j(\gamma),\ \text{for all}\ n\in
 N$$
$$p\phi^{nk_1m}(c)\ \text{is homotopic to}\ q\phi^{-nk_2m}(c^2),\ \text{for all}\ n\in N$$

By using the same argument, we know $q\phi^{-nk_2m}(c)$ is
homotopic to a simple closed curve on $F$, for all $n\in N$. Let
$\alpha_n$ denote the geodesics in the free homotopy class of
$\phi^{-nk_2m}(c)$, there exists a subsequence of $\alpha_n$,
without loss of generality, still call it $\alpha_n$, such that
$\alpha_n\rightarrow \Lambda^u$ as $n\rightarrow \infty$. Since
$q\phi^{-nk_2m}(c)$ is homotopic to a simple closed curve on $F$
for all n, the geodesics in the free homotopic classes of
$q\phi^{-nk_2m}(c)$ converge to a geodesic lamination
$\Theta\subset F$, by passing to a subsequence. It follows that
$q(\Lambda^u)$ is a geodesic lamination.

Notice that in the proof, we can only lift $\phi^m: S\rightarrow
S$ by $i$ and $j$ to homeomorphisms of $G$ for some $m\in N$, but
the end points of $\partial\widetilde p_i(\widetilde\Lambda^s)$
and $\partial\widetilde q_j(\widetilde\Lambda^u)$ for any $i,j\in
\{1,2\}$ do not depend on $m$. Therefore we have proved that
$\{\partial\widetilde p_1(\widetilde\Lambda^s),\partial \widetilde
p_2(\widetilde \Lambda^s),\partial \widetilde
q_1(\widetilde\Lambda^u),$ $\partial \widetilde
q_2(\widetilde\Lambda^u)\}$ is a pairwise disjoint set. According
to Theorem 6, we know $\pi_1(\mathcal G_{\phi^m})$ is hyperbolic
for sufficiently large $m$.
\end{proof}
\bigskip

\section{An example which is not abstractly commensurate to a surface-by-free group}

In this section, we will show that there exist a graph of
surfaces whose fundamental group is hyperbolic, but is not
abstractly commensurate to any surface-by-free group, for any
closed hyperbolic surface or orbifold $S^{\prime}$ and any free group $K$. Therefore
this group is different from all the groups constructed in
\cite{Mosher:hypbyhyp}. By applying Theorem 1.1 in
\cite{FarbMosher:sbf}, it follows that the example constructed
here is not even quasi-isometric to any surface-by-free group.

Recall that, groups $G$ and $H$ are called \emph{abstractly
commensurate}, if there exist finite index subgroups $G_1<G$ and
$H_1<H$, so that $G_1$ is isomorphic to $H_1$. A group $G$ is
called a \emph{surface-by-free} group, if there is a
hyperbolic surface or a hyperbolic orbifold $S$, and a free group
$K$, such that there exists a short exact sequence:

$$1\rightarrow \pi_1(S)\rightarrow G\rightarrow K\rightarrow 1 $$

First, we shall give a necessary and sufficient condition for a
group to be abstractly commensurate to a surface-by-free group.
Second, we shall construct a non-hyperbolic graph of surfaces
$\mathcal G$, by applying the condition, whose fundamental group
is not abstract commensurate to any surface-by-free group.
Finally, we shall construct a hyperbolic graph of surfaces
$\mathcal G_{\phi^m}$ from $\mathcal G$ such that $\pi_1(\mathcal
G_{\phi^m})$ is not abstractly commensurate to any surface-by-free
group.

Let $t$ denote the Bass-Serre tree of a graph of surfaces
$S\Gamma$, and let $V,\ E$ denote the set of all the vertices and
edges of $t$ respectively. $\pi_1(S\Gamma)$ acts on $t$ with
subgroups $stab(v)$ and $stab(e)$, which stabilize the vertex
$v\in V$ and the edge $e\in E$ respectively.

\begin{lemma}The fundamental group of a graph of surfaces $S\Gamma$ is
 abstractly commensurate to a surface-by-free group if and only if
$[stab(v):\cap_{w\in V}stab(w)]<\infty$, for any $v\in V$.
\end{lemma}
\begin{proof}
According to \cite{FarbMosher:sbf}, a finite index subgroup of a
surface-by-free group is a surface-by-free group. If
$\pi_1(S\Gamma)$ is abstractly commensurate to a surface-by-free
group, then there exists a finite index subgroup of $H$ of
$\pi_1(S\Gamma)$ which is isomorphic to a surface-by-free group.

$H$ acts on $t$, and $[stab(v):H\cap stab(v)]\leq
[\pi_1(S\Gamma):H]$ is finite. $H$ acts on $t$ with compact
quotient, $t$ may be identified with the Bass-Serre tree of $H$.
Since $H$ is isomorphic to a surface-by-free group
$\pi_1(S^{\prime})\rtimes F$, where $S^{\prime}$ is a hyperbolic
surface, $F$ is a finite rank free group, there exists a normal
subgroup $N$ of $H$ which is isomorphic to $\pi_1(S^{\prime})$,
such that $N$ acts trivially on $t$.

Let $N$ denote $\cap_{w\in V}(stab(w)\cap H)$ which is a finite
index subgroup of $stab(v)\cap H$ for any vertex $v\in t$, i.e.,
$[stab(v)\cap H:\cap_{w\in V}(stab(w)\cap H)]<\infty$. Therefore:
\begin{align*}
 &[stab(v):\cap_{w\in V}stab(w)]<[stab(v):\cap_{w\in
V}(stab(w)\cap H)]\\
&=[stab(v):H\cap stab(v)][H\cap stab(v):\cap_{w\in V}(stab(w)\cap
H)]<\infty.
\end{align*}

We have finished the proof for one direction.

Now we will prove the other direction. The action of
$\pi_1(S\Gamma)$ on $t$ induces a homomorphism $\sigma:
\pi_1(S\Gamma)\rightarrow Aut(t)$. Let $K=\cap_{w\in V}stab(w)$,
$K=ker(\sigma)$. Since $K$ is a finite index subgroup of $stab(v)$
for any $v\in V$, $\pi_1(S\Gamma)/K$ acts on $t$ with finite edge
and vertex stabilizers. In addition $\pi_1(S\Gamma)/K$ acts on $t$
cocompactly. Therefore $t/(\pi_1(S\Gamma)/K)$ is a finite graph of
finite groups. Applying Theorem 7.3 in \cite{ScottWall}, it
follows that $\pi_1(S\Gamma)/K$ is virtually free. Hence
$\pi_1(S\Gamma)$ is abstractly commensurate to a surface-by-free
group.
\end{proof}

In the rest of this paper, let $\mathcal G$ denote a graph of
surfaces as in Figure 5,
\begin{figure}[h]
\begin{center}
\includegraphics{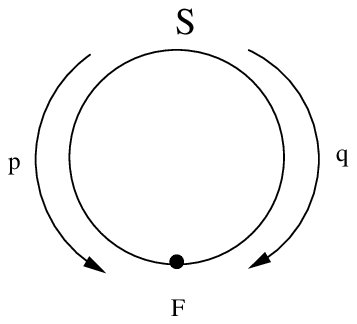}
\end{center}
\caption{}
\end{figure}
where $S$, $F$, $p$, $q$ and the simple closed curves $c\subset
S$, $a\subset F$ are as described in Corollary \ref{hnn}.

The conclusion of the following lemma that $\pi_1(\mathcal G)$ is
not commensurate to a surface-by-free group was discovered and
proved independently by Chris Odden in his thesis, and by Lee
Mosher. I will give a different proof which will generalize to my
later examples.

Define subgroups $L_i$, $R_i$, $G_i$ of $\pi_1(S)$, and subgroups
$H_i$ of $\pi_1(F)$ by induction as follows:

\begin{itemize}

\item let $H_0=\pi_1(F)$, $G_0=\pi_1(S)$,

\item let $H_1=p_{\ast}(G_0)\cap q_{\ast}(G_0)$,
$L_1=p_{\ast}^{-1}(H_1)$, $R_1=q_{\ast}^{-1}(H_1)$, $G_1=L_1\cap
R_1$,

\item let $H_{i+1}=p_{\ast}(G_i)\cap q_{\ast}(G_i)$,
$L_{i+1}=p_{\ast}^{-1}(H_{i+1})$,
$R_{i+1}=q_{\ast}^{-1}(H_{i+1})$, $G_{i+1}=L_{i+1}\cap R_{i+1}$.

\end{itemize}

 From $p_{\ast}([c])=q_{\ast}([c^2])=[a^2]$, we know $[a^2]\in
H_1$ but $[a]\notin H_1$,  $[c]\in L_1$, $[c^2]\in R_1$ but
$[c]\notin R_1$.

Therefore $L_1\neq R_1$, $[c^2]\in G_1$, $[c]\notin G_1$.

Similarly, from $p_{\ast}([c^2])=[a^4]$, $q_{\ast}([c^2])=[a^2]$,
we know $p_{\ast}(G_1)\neq q_{\ast}(G_1)$, $[c^2]\in L_2$,
$[c^4]\in R_2$, but $[c^2]\notin R_2$, $[c^4]\in G_2$.

Inductively, we have $[c^{2^n}]\in G_n$,
$p_{\ast}([c^{2^n}])=[a^{4^n}]$, $q_{\ast}([c^{2^n}])=[a^{2^n}]$,

\noindent so $p_{\ast}(G_n)\neq q_{\ast}(G_n)$,  $[c^{2^n}]\in
L_{n+1}$, $[c^{2^n}]\notin R_{n+1}$, but $[c^{2^{n+1}}]\in
R_{n+1}$.

Hence we get two sequences $\{L_i\}$ and $\{R_i\}$ of finite index
normal subgroups of $\pi_1(S)$, the indexes of $[\pi_1(S):L_i]$
and $[\pi_1(S):R_i]\rightarrow \infty$ as $i\rightarrow \infty$.

\begin{lemma}
Suppose the edge group $\pi_1(S)$ of $\pi_1(\mathcal G)$ contains
two nested sequences of finite index normal subgroups
$L_1>L_2>\cdots$ and $R_1>R_2>\cdots$ which are constructed
inductively as follows:
\begin{enumerate}
\item $H_1=p_{\ast}(\pi_1(S))\cap q_{\ast}(\pi_1(S))$,
$L_1=p_{\ast}^{-1}(H_1)$, $R_1=q_{\ast}^{-1}(H_1)$, $G_1=L_1\cap
R_1$

\item $H_{i+1}=p_{\ast}(G_i)\cap q_{\ast}(G_i)$

 \item $L_{i+1}=p_{\ast}^{-1}(H_{i+1})$, $R_{i+1}=q_{\ast}^{-1}(H_{i+1})$, $G_{i+1}=L_{i+1}\cap R_{i+1}$

\end{enumerate}
If  $L_i\neq R_i$ for all $i$, then $\pi_1(\mathcal G)$ is not
commensurate to a surface-by-free group.
\end{lemma}
\begin{proof}
It is known that every edge or vertex stabilizer in the Bass-Serre
tree $t$ is isomorphic to some edge or vertex group of the graph
of groups. Let $e_1$ be an edge of the Bass-Serre tree $t$ such
that the stabilizer $stab(e_1)=\pi_1(S)$. Let $g$ be the generator
of the underlying graph $\Gamma$ of the graph of spaces $\mathcal
G$; \cite{ScottWall} says that if $\pi_1(S)$ is identified with
$p_{\ast}(\pi_1(S))$, then $q_{\ast}(\pi_1(S))=g^{-1}\pi_1(S)g$.
 There exists a unique edge $e_2\in t$, such that $e_2=ge_1$. It is easy to see that $R_1=q^{-1}_{\ast}(p_{\ast}(\pi_1(S)\cap
q_{\ast}(\pi_1(S)))=stab(e_1)\cap stab(e_2)$. Let $e_j=ge_{j-1}$
for a positive integer $j$, let $\alpha_i$ be the oriented path
$e_1\ast \cdots \ast e_i$ in the Bass-Serre tree $t$.
$\cap_{\epsilon\in \alpha_i} stab(\epsilon)=\cap^i_{j=1}
stab(e_j)=R_i$. Similarly, there exists another sequence of
oriented paths $\{\beta_{k}\}$ in $t$ such that $\cap_{\epsilon\in
\beta_{k}}stab(\epsilon)=L_{k}$. Therefore
$[\pi_1(S):L_i]\rightarrow \infty$ and $[\pi_1(S):R_i]\rightarrow
\infty$ imply $[stab(e):\cap_{\epsilon \in
E}stab(\epsilon)]=\infty$. For the case studied here, every edge
stabilizer is a finite index subgroup of some vertex stabilizers,
if the vertex is an end point of that edge. So
$[stab(e):\cap_{\epsilon \in E}stab(\epsilon)]=\infty$ implies
$[stab(v):\cap_{w \in V}stab(w)]=\infty$. According to Lemma 10,
$\pi_1(\mathcal G)$ is not commensurate to a surface-by-free
group.
\end{proof}

\bigskip

In order to construct a group which is not abstractly commensurate
to a surface-by-free group, our first strategy is to find a
pseudo-Anosov mapping class $\Phi$ which fixes all the finite
index normal subgroups of $\pi_1(S)$. But unfortunately, the
theorem below tells us that there does not exist such a
pseudo-Anosov mapping class.

\begin{theorem}
Let $S_n$ be a closed surface with genus $n$, where $n\geq 2$. For
any $\Phi \in Aut(\pi_1(S_n))$, if $\Phi$ fixes all the finite
index normal subgroups of $\pi_1(S_n)$, then $\Phi \in
Inn(\pi_1(S_n))$.

\end{theorem}

Before proving this theorem, we introduce some related history and
preliminaries first.

In \cite{lubotzky:normaautoofreegp}, Lubotzky proved that for any
free group $F_n$, $n\geq 2$, if $\Psi\in Aut(F_n)$ fixes all the
finite index normal subgroups of $F_n$, then $\Psi\in Inn(F_n)$.
In particular, every normal automorphism of $F_n$ is inner.
Bogopolski, Kudryavtseva and Zieschang in
\cite{bkz:scurveonsurfaces} proved that for any closed hyperbolic
surface $S_n$ with genus $n$ not less than 2, if $\Phi\in
Aut(\pi_1(S_n))$ fixes all the normal subgroups of $\pi_1(S_n)$,
then $\Phi\in Inn(\pi_1(S_n))$. The main theorem in that paper
says for any non-separating simple closed curve $\alpha$ on $S$,
up to conjugate equivalent, $\alpha$ is the only non-separating
simple closed curve in its normal closure. The theorem in
\cite{bkz:scurveonsurfaces} says:
\begin{theorem}
Let $S$ be a closed orientable surface and $g$, $h$ are
non-trivial elements of $\pi_1(S)$ both containing simple closed
two-sided curves $\gamma$ and $\kappa$, resp. The group element
$h$ belongs to the normal closure of $g$ if and only if $h$ is
conjugate to $g^{\epsilon}$ or to $(gug^{-1}u^{-1})^{\epsilon}$,
$\epsilon\in \{1,-1\}$; here $u$ is a homotopy class containing a
simple closed curve $\mu$ which properly intersects $\gamma$
exactly once.
\end{theorem}

I would like to thank Jason Deblois for help with Lemma 14.
\begin{lemma}
For any two non trivial, non freely homotopic, non-separating simple
closed curves $a$ and $b$ on S, let $[a],[b]$ denote the homotopy
class of them in $\pi_1(S)$. There exists a finite index normal
subgroup $N\in \pi_1(S)$, such that $[a]\in N$ and $[b]\notin N$.

\end{lemma}

A group $G$ is said to be \emph{residually finite}, if for any
element $g\in G$, $g\neq 1$, there exists a finite group $K$ and a
homomorphism $h:G\rightarrow K$, such that $h(g)\neq 1$.

A \emph{Haken manifold} is a compact, orientable, irreducible
3-manifold which contains a 2-sided incompressible surface.

\begin{proof}: Let $M=S\times I$, where $I$ is the interval $[0,1]$.
$\pi_1(M)$ is isomorphic to $\pi_1(S)$. Since $a$ is a simple
closed curve on $S$, attach a 2-handle $B$ to $M$ along $a\times
\{0\} \cup a\times \{1\}$ obtain a Haken manifold $M^{\prime}$.
This attachment gives a surjective homomorphism $\epsilon:
\pi_1(M)\rightarrow \pi_1(M^{\prime})$, and the kernel is the
normal closure of $[a]$. Since $a$ is the only non-separating
simple closed curve in the normal closure of $[a]$, by applying
Theorem 13, it follows that $[b]$ does not belong to the kernel of
$\epsilon$.

According to Theorem 1.1 in \cite{Hampel:residualfinite3mfld},
$\pi_1(M^{\prime})$ is residually finite. So for $[b]\in
\pi_1(M)$, there exist a finite group $K$ and a homomorphism
$\delta:\pi_1(M^{\prime})\rightarrow K$, such that $[b]\notin
ker(\delta)$.

Let $N$ denote the kernel $ker(\delta\circ\epsilon)$. Obviously,
$N$ is a finite index normal subgroup of $\pi_1(S)$, and $[a]\in
N$, but $[b]\notin N$.

\end{proof}

\begin{proof}[Proof of Theorem 12:] Let $\Phi$ be an element of $Aut(\pi_1(S))$, and let $\phi$
be a representative of it in $Homeo(S)$. According to
\cite{bkz:scurveonsurfaces}, if $\Phi\notin Inn(\pi_1(S))$, then
there exists a non-separating simple closed curve $a$ on $S$, such
that $a$ and $\phi(a)$ are not freely homotopic to each other.
According to Lemma 14, there exist a finite index normal subgroup
$N\lhd \pi_1(S)$, such that $[a]\in N$ and $[\phi(a)]\notin N$. It
follows that $\Phi(N)\neq N$.
\end{proof}
\bigskip

In the following, we shall construct a pseudo-Anosov mapping class
which does not fix all the finite index normal subgroups of
$\pi_1(S)$, but fixes $L_i$ and $R_i$ as in Lemma 11.

In the following, let $\mathcal G_{\phi^m}$ be a graph of surfaces
as in Figure 3, where $F,\ S,\ p,\ q$ as described in Corollary
\ref{hnn}.
\begin{theorem}
There exists a pseudo-Anosov homeomorphism $\phi\in Homeo(S)$, so
that $\pi_1(\mathcal G_{\phi^m})$ is hyperbolic but is not
commensurate to a surface-by-free group.
\end{theorem}

\begin{figure}[h]
\begin{center}
\input{torus3_m.pstex_t}
\end{center}
\caption{}
\end{figure}

\begin{figure}[h]
\begin{center}
\input{torus2_m.pstex_t}
\end{center}
\caption{}
\end{figure}

\begin{proof}
 If there exists a pseudo-Anosov
homeomorphism $\phi$, such that $\phi_{\ast}(L_i)=L_i$ and
$\phi_{\ast}(R_i)=R_i$, according to Lemma 11, then $[stab(e):\cap
stab_{\epsilon \in E}(\epsilon)]=\infty$. Therefore
$\pi_1(\mathcal G_{\phi^m})$ is not commensurate to a
surface-by-free group.

Curves mentioned in this theorem are shown in Figure 6 and Figure
7, these two figures are a refinement of Figure 4.

First, we will describe the covering maps $p$ and $q$ with more
details. Let $p^{-1}(a^2)=c$, $q^{-1}(a)=c\cup d$. It is easy to
see that $p(\alpha)$ is homotopic to $q(\alpha)\subset F$ and
$p(\beta)$ is homotopic to $q(\beta)\subset F$, where $\alpha,\
\beta\subset S$ as in Figure 6. Therefore $[\alpha],\ [\beta]\in
L_i\cap R_i$ for all $i$.

Second, we claim that if $\gamma$ is a simple closed curve in $S$,
such that $[\gamma]\subset L_i$ for some $i$, then
$(\tau_{\gamma})_{\ast}$ of the Dehn-twist $\tau_{\gamma}$ fixes
$L_i$. Note that $L_i$ is a finite index normal subgroup of
$\pi_1(S)$ if and only if there exists a finite group $K$ and a
homomorphism $f:\pi_1(S)\rightarrow K$ such that $L_i=ker(f)$. We
shall see that $\tau_{\gamma}$ maps every element in the kernel of
$f$ to an element in the kernel of $f$, i.e., $(\tau)_{\ast}$
fixes $L_i$. Let $[g]\in \pi_1(S)$ be an element in $L_i$.
Decompose $[g]=[h_1]\cdots[h_n]$, so that $[h_j]\in \pi_1(S)$ is
represented by a closed loop in $S$ which has only one transverse
intersection point with $\gamma$, for all $j\in \{1,\cdots,n\}$.
Depending on how $h_i$ intersects with $\gamma$,
$[\tau_{\gamma}(h_i)]$ is one of the following four kinds:
$[h_i\gamma]$, $[\gamma h_i]$, $[\gamma h_i\gamma^{-1}]$,
$[\gamma^{-1} h_i\gamma]$. The trivial case is $g\cap
\gamma=\emptyset$, so $f([\tau_{\gamma}(g)])=[g]$. Otherwise,
\begin{align*}
f([\tau_{\gamma}(g)])&=f([\tau_{\gamma}(h_1)]\cdots [\tau_{\gamma}(h_n)])\\
&=f([\tau_{\gamma}(h_1)])\cdots f(\tau_{\gamma}(h_n))\\
&=f([h_1])\cdots f([h_n])=f([h_1\cdots h_n])=f([g])=I_k\\
\intertext{where $I_k$ is the identity of $K$. It shows that $(\tau_{\gamma})_{\ast}$ fixes $L_i$.}
\end{align*}

If we can find disjointly essential
 curve systems $\mathcal C$ and $\mathcal D$ which satisfy the
conditions in Theorem 3, and if all the homotopy classes of the
elements of $\mathcal C$ and $\mathcal D$ belong to $L_i$ and
$R_i$ for all $i$, then we can construct a pseudo-Anosov
homeomorphism $\phi$ as described in Theorem 3, such that
$\phi_{\ast}$ fixes $L_i$ and $R_i$ for all $i$.

In the following, we will prove that there exist disjointly essential
 curve systems $\mathcal C=
\alpha \cup \widehat{\alpha}$, and $\mathcal D=\beta \cup
\widehat{\beta}$, such that $\mathcal C\ \cup \mathcal D$ fills
$S$, where $\alpha$, $\beta$ as in Figure 6. In addition,
$[\alpha]$, $[\widehat \alpha]$, $[\beta]$ and $[\widehat
\beta]\in \cap_i(L_i\cap R_i)$.

In order to find a simple closed curve $\widehat \alpha$
satisfying the above conditions, first, we will show that there
exists a simple closed curve $\alpha^{\prime}$ such that
$[\alpha^{\prime}]\in \cap_i(L_i\cap R_i)$. Since $L_i$ and $R_i$
are finite index normal subgroups of $\pi_1(S)$ , and $[\alpha]\in
\cap_i(L_i\cap R_i)$, the normal closure $N_{\alpha}$ of
$[\alpha]$ is a subgroup of $\cap_i(L_i\cap R_i)$. Recall that the
normal closure $N_{\alpha}$ of $[\alpha]$ is the smallest normal
subgroup of $\pi_1(S)$ contains $[\alpha]$. Applying Theorem 13,
we only need the easy direction of this theorem, the separating
curve $\alpha^{\prime}$ as in Figure 6 represents an element in
$N_{\alpha}$.

Second, we shall construct a simple closed curve $\widehat\alpha$
on $S$ from the simple closed curve $\alpha^{\prime}$.

From \cite{Mosher:traintracks}, we know there exists a short exact
sequence:
$$1\rightarrow \langle T_{\alpha}\rangle\rightarrow stab(\alpha)\rightarrow MCG(S-\alpha)\rightarrow 1$$
where $\langle T_{\alpha}\rangle$ is a cyclic subgroup of $MCG(S)$
generated by the mapping class $T_{\alpha}$ of the Dehn-twist
$\tau_{\alpha}$ around $\alpha$, $stab(\alpha)$ is a subgroup of
$MCG(S)$ which fixes $\alpha$, $S-\alpha$ is a surface by cutting
$S$ along $\alpha$. The homomorphism $\iota:
stab(\alpha)\rightarrow MCG(S-\alpha)$ is defined as
$\Phi\rightarrow\Phi|_{S-\alpha}$, for $\Phi\in stab(\alpha)$.

Choose a pseudo-Anosov homeomorphism $\psi\in Homoeo(S-\alpha)$,
maybe need pass to a high enough power of $\psi$, such that
$\widehat{\alpha}=\psi(\alpha^{\prime})$ is very close to the
stable geodesic lamination $\Lambda^S_{\psi}$ of $\psi$, therefore
$\widehat{\alpha}\cup \beta$ fills $S-\alpha$. Also
$\widehat{\alpha}$ is disjoint with $\alpha$ because
$\alpha^{\prime}$ is disjoint with $\alpha$.

Using the same method, we can find a simple closed curve
$\widehat{\beta}$ which is disjoint with $\beta$ and
$\widehat{\beta}\cup \alpha$ fills $S-\beta$.

Let $\mathcal C=\{\alpha, \widehat{\alpha}\}$, $\mathcal
D=\{\beta, \widehat{\beta}\}$, it is easy to see that $\mathcal
C\cup \mathcal D$ fills $S$. According to Theorem 3, if $\phi_0$
is a homeomorphism of $S$, such that $\tau_{\alpha}^{+}$,
$\tau_{\widehat{\alpha}}^{+}$, $\tau_{\beta}^{-}$ and
$\tau_{\widehat{\beta}}^{-}$ appear at least once in $\phi_0$,
then $\phi_0$ is a pseudo-Anosov homeomorphism. Since
$[\alpha]$,$[\\widehat{\alpha}]$, $[\beta]$, $[\widehat{\beta}]\in
\cap_i(L_i\cap R_i)$,  $(\phi_0)_{\ast}$ fixes $L_i$ and $R_i$ for
all i.

In order to finish the proof of this theorem, according to
Corollary \ref{hnn}, we only need to show that there exists some
pseudo-Anosov homeomorphism $\phi$ constructed as above, so that
the virtual centralizer $VC\langle \Phi \rangle$ of $\langle \Phi
\rangle$ has trivial intersection with the mapping classes of the
deck transformation groups of the covering maps $p$ and $q$
respectively, where $\Phi\in MCG(S)$ is the mapping class of
$\phi$. Abusing of notation, denote both the deck transformations
and the mapping classes of the deck transformations by $D_p$ and
$D_q$. The deck transformation group of $p$ has only two elements
$D_p$ and the identity.

Let $\phi_0$ be a pseudo-Anosov homeomorphism of $S$ constructed
above, and let $\Phi_0$ be its mapping class. Let
$\Lambda_{\phi_0}^s$ and $\Lambda_{\phi_0}^u$ be the stable and
unstable geodesic laminations of $\phi_0$ respectively. It is
known that $\Phi_0$ fixes $L_i$ and $R_i$ for all $i$.

Suppose the deck transformation group of $p$ has nontrivial
intersection with the virtual centralizer of
$\langle\Phi_0\rangle$, i.e.,
$D_p(\Lambda_{\phi_0}^s)=\Lambda_{\phi_0}^s$. We claim that
$D_p(T_{\alpha}(\Lambda_{\phi_0}^s))\neq
T_{\alpha}(\Lambda_{\phi_0}^s)$, where $T_{\alpha}$ is the mapping
class of the Dehn-twist $\tau_{\alpha}$. Notice that
$T_{\alpha}(\Lambda_{\phi_0}^s)$ is the stable geodesic lamination
of the pseudo-Anosov mapping class $T_{\alpha}\Phi_0
T_{\alpha}^{-1}$, and $T_{\alpha}\Phi_0 T_{\alpha}^{-1}$ fixes
$L_i$ and $R_i$ for all $i$. If the claim is true, let
$\Phi_1=T_{\alpha}\Phi_0 T_{\alpha}^{-1}$, then
$VC\langle\Phi_1\rangle$ has trivial intersection with $D_p$.

 We shall prove the claim. Notice that there exists a simple closed curve
$\gamma$ on $S$ is disjoint with $\alpha$, such that
$D_p(\alpha)=\gamma$. According to Lemma 4.1.C in
\cite{Ivanov:mcg},
$D_pT_{\alpha}D_p^{-1}=T_{D_p(\alpha)}=T_{\gamma}$.

Suppose
$D_pT_{\alpha}(\Lambda_{\phi_0}^s)=T_{\alpha}(\Lambda_{\phi_0}^s)$,
then:
\begin{align*}
D_pT_{\alpha}(\Lambda_{\phi_0}^s)&=D_pT_{\alpha}\Phi_0T_{\alpha}^{-1}(T_{\alpha}(\Lambda_{\phi_0}^s))\\
&=T_{\gamma}D_p\Phi_0T_{\alpha}^{-1}(T_{\alpha}(\Lambda_{\phi_0}^s))\\
&=T_{\gamma}D_p\Phi_0(\Lambda_{\phi_0}^s)=T_{\gamma}(\Lambda_{\phi_0}^s)\\
\intertext{Therefore
$T_{\alpha}(\Lambda_{\phi_0}^s)=T_{\gamma}(\Lambda_{\phi_0}^s)$.}
\end{align*}

It follows that $T_{\alpha}^{-1}T_{\gamma}\in
VC\langle\Phi_0\rangle$, but from Theorem 3.5 in
\cite{Mosher:Geosurvmcg}, we know that $VC\langle\Phi_0\rangle$
has $\langle\Phi_0\rangle$ as a finite index subgroup. Hence up to
some power $m$, $(T_{\alpha}^{-1}T_{\gamma})^m\in
\langle\Phi_0\rangle$, but obviously
$(T_{\alpha}^{-1}T_{\gamma})^m$ is neither pseudo-Anosov nor the
identity, so it is not an element of $\langle\Phi_0\rangle$.
Therefore $D_pT_{\alpha}(\Lambda_{\phi_0}^s)\neq
T_{\alpha}(\Lambda_{\phi_0}^s)$.

If in addition  $D_qT_{\alpha}(\Lambda_{\phi_0}^u)\neq
T_{\alpha}(\Lambda_{\phi_0}^u)$, then take $\Phi=\Phi_1$, this theorem is proved.

If
$D_qT_{\alpha}(\Lambda^u_{\phi_0})=T_{\alpha}(\Lambda_{\phi_0}^u)$,
then we claim $D_qT_{\alpha}^2(\Lambda_{\phi_0}^u)\neq
T^2_{\alpha}(\Lambda^u_{\phi_0})$. If the claim is not true, then
\begin{align*}
D_qT^2_{\alpha}(\Lambda^u_{\phi_0})&=T^2_{\alpha}(\Lambda^u_{\phi_0})\\
&=T_{\alpha}(D_qT_{\alpha}(\Lambda^u_{\phi_0}))\\
&=D_qT_{\theta}(T_{\alpha}(\Lambda^u_{\phi_0})),\\
\end{align*}
where $\theta=D_q(\alpha)$ is a simple closed curve on $S$
disjoint from $\alpha$. Therefore
$$T_{\alpha}^{-1}T^{-1}_{\theta}D_q^{-1}D_qT_{\alpha}^2(\Lambda^u_{\phi_0})=\Lambda^u_{\phi_0}$$
It follows that
$T_{\alpha}^{-1}T_{\theta}^{-1}T_{\alpha}^2(\Lambda^u_{\phi_0})=\Lambda^u_{\phi_0}$.
Since $\theta$, $\alpha$ are disjoint simple closed curves,
$T_{\alpha}T_{\theta}^{-1}=T_{\theta}^{-1}T_{\alpha}$. Hence
$T^{-1}_{\alpha}T_{\theta}^{-1}T_{\alpha}^2(\Lambda^u_{\phi_0})=T_{\theta}^{-1}T_{\alpha}(\Lambda^u_{\phi_0})=\Lambda^u_{\phi_0}$.
By the same reason in the above argument, it is impossible.

Replacing $T_{\alpha}$, $T_{\gamma}$ by $T_{\alpha}^2$,
$T_{\gamma}^2$ in the above proof of
$D_pT_{\alpha}(\Lambda^s_{\phi_0})\neq
T_{\alpha}(\Lambda^s_{\phi_0})$, we can see
$D_pT_{\alpha}^2(\Lambda_{\phi_0}^s)\neq
T_{\alpha}^2(\Lambda_{\phi_0}^s)$. Take
$\Phi=T_{\alpha}^2\Phi_0T_{\alpha}^{-2}$, then this theorem is
proved.
\end{proof}

\newpage
\addcontentsline{toc}{section}{\protect\numberline{}{References}}

\newcommand{\etalchar}[1]{$^{#1}$}
\providecommand{\bysame}{\leavevmode\hbox to3em{\hrulefill}\thinspace}
\providecommand{\MR}{\relax\ifhmode\unskip\space\fi MR }
\providecommand{\MRhref}[2]{%
  \href{http://www.ams.org/mathscinet-getitem?mr=#1}{#2}
}
\providecommand{\href}[2]{#2}

\end{document}